\documentclass[11pt]{article}

\usepackage{lineno,hyperref}
\usepackage{rotating}
\usepackage{makeidx}
\usepackage{float}
\usepackage{epsfig}
\usepackage{amsmath,amssymb, amsthm}
\usepackage[latin1]{inputenc}
\usepackage{epic,eepic}
\usepackage{overpic}
\usepackage{color}
\usepackage{amsfonts}
\usepackage{graphicx}
\usepackage{pgf,tikz,pgfplots}
\usepackage{mathrsfs}
\usetikzlibrary{arrows}
\usepackage{mwe} 
\usepackage{subcaption}

\usepackage{algorithm}
\usepackage{algorithmicx}
\usepackage{algpseudocode}
\usepackage{multirow}

            {\hfill\penalty10000\raisebox{-.09em}{\large\bf\rm $\blacksquare$}\par\medskip}

\newtheorem{theorem}{Theorem}[section]
\newtheorem{definition}{Definition}[section]
\newtheorem{lemma}[theorem]{Lemma}
\newtheorem{proposition}[theorem]{Proposition}
\newtheorem{corollary}[theorem]{Corollary}

\newcommand\scalemath[2]{\scalebox{#1}{\mbox{\ensuremath{\displaystyle #2}}}}

\makeatletter
\newcommand{\TITLE}[2][]{
  \begin{center} \Large \bfseries #2%
  \def\@temp{#1}\ifx\@temp\@empty\relax\else\footnote{#1}\fi\end{center}}

\newcommand{\AUTHOR}[2][]{\textbf{#2}%
  \def\@temp{#1}\ifx\@temp\@empty\relax\else\textsuperscript{#1}\fi}

\newcommand{\ADDRESS}[2]{\noindent
  \textsuperscript{#1}\parbox[t]{0.5\textwidth}{#2} \vspace{6pt}}

\usepackage{lipsum}
\makeatletter
\def\ps@pprintTitle{%
 \let\@oddhead\@empty
 \let\@evenhead\@empty
 \def\@oddfoot{}%
 \let\@evenfoot\@oddfoot}
\makeatother

\makeatother
\newcommand{\KEYWORDS}[1]{\paragraph{Keywords:} #1}
\newcommand{\AMSCLASSIFICATION}[1]{\paragraph{Mathematics Subject Classification 2000:} #1}

\begin{document}

\TITLE[This work was funded by project 20928/PI/18 (Proyecto financiado por la Comunidad Aut\'onoma de la Regi\'on de Murcia a trav\'es de la convocatoria de Ayudas a proyectos para el desarrollo de investigaci\'on cient\'ifica y t\'ecnica por grupos competitivos, incluida en el Programa Regional de Fomento de la Investigaci\'on Cient\'ifica y T\'ecnica (Plan de Actuaci\'on 2018) de la Fundaci\'on S\'eneca-Agencia de Ciencia y Tecnolog\'ia de la Regi\'on de Murcia), by the national research project MTM2015- 64382-P (MINECO/FEDER), by AFOSR grant FA9550-20-1-0055, by NSF grant DMS-1719410 and by Spanish MINECO project MTM2017-83942-P]{A new WENO-2r algorithm with progressive order of accuracy close to discontinuities}

\author{
Sergio Amat\thanks{
 Departamento de Matem\'atica Aplicada y Estad\'{\i}stica.
   Universidad  Polit\'ecnica de Cartagena (Spain).
e-mail:{\tt sergio.amat@upct.es}}\and
Juan Ruiz \thanks{ Departamento de Matem\'atica Aplicada y Estad\'{\i}stica.
   Universidad  Polit\'ecnica de Cartagena (Spain).
e-mail:{\tt juan.ruiz@upct.es}}
\and Chi-Wang Shu\thanks{Division of Applied Mathematics. Brown University. Providence, Rhode Island, USA.
e-mail:{\tt chi-wang\_shu@brown.edu}}
\and Dionisio F. Y\'a\~nez\thanks{Departamento de Matem\'aticas, Facultad de Matemáticas. Universidad de Valencia. Valencia, Spain.
e-mail:{\tt dionisio.yanez@uv.es}}
}

\begin{center}
 \AUTHOR[1]{Sergio Amat},
\AUTHOR[2]{Juan Ruiz},
\AUTHOR[3]{Chi-Wang Shu},
\AUTHOR[4]{Dionisio F. Y\'a\~nez}

\end{center}

\AMSCLASSIFICATION{41A05, 41A10, 65D05, 65M06, 65N06}
\vspace{1cm}

\ADDRESS{1}{Department of Applied Mathematics and Statistics, Universidad Polit\'ecnica de Cartagena (UPCT) (Spain).
\texttt{sergio.amat@upct.es}}
 \quad
 \ADDRESS{2}{Department of Applied Mathematics and Statistics, Universidad Polit\'ecnica de Cartagena (UPCT) (Spain).
\texttt{juan.ruiz@upct.es}}

 \ADDRESS{3}{Division of Applied Mathematics. Brown University. Providence, Rhode Island, USA.
\texttt{chi-wang\_shu@brown.edu}}
 \quad
\ADDRESS{4}{Departamento de Matem\'aticas, Facultad de Matemáticas. Universidad de Valencia. Valencia, Spain.
\texttt{dionisio.yanez@uv.es}}

\small
\begin{abstract}
In this article we present a modification of the algorithm for data discretized in the point values introduced in [S. Amat, J. Ruiz, C.-W. Shu, On a new WENO algorithm of order $2r$ with improved accuracy close to discontinuities, App. Math. Lett. 105 (2020), 106-298]. In the aforementioned work, we managed to obtain an algorithm that reaches a progressive and optimal order of accuracy close to discontinuities for WENO-6. For higher orders, i.e. WENO-8, WENO-10, etc. We have found that the previous algorithm presents some {\it shadows} in the detection of discontinuities, meaning that the order of accuracy is better than the one attained by WENO of the same order, but not optimal. In this article we present a modification of the smoothness indicators used in the original algorithm, oriented to solve this problem and to attain a WENO-2r algorithm with progressive order of accuracy close to the discontinuities. We also present proofs for the accuracy and explicit formulas for all the weights used for any order $2r$ of the algorithm.
\end{abstract}

\section{Introduction: Classical WENO algorithm}\label{introduction}
WENO (weighted essentially non oscillatory) algorithm \cite{Liu, JiangShu, doi:10.1137/100791579, AMB, Henrick2005542, Castro20111766, Shu1998, Shu1999, doi:10.1137/070679065, Shu2016} was designed to use the stencil of ENO (essentially non oscillatory)  algorithm \cite{MR881365, Harten1987231} and to behave in a similar way close to discontinuities, while improving the accuracy at smooth zones. WENO is written using a convex combination of all the interpolating polynomials that share the central interval of the global stencil used. In order to provide an adaptive approximation, the weights of the combination are nonlinear and based on an efficient estimation of the smoothness of each sub-stencil using what are called {\it smoothness indicators} \cite{JiangShu}. In \cite{generalizacion} we proposed a generalization of the algorithm introduced in\cite{WENO_nuevo}, where we aimed to provide a WENO-6 algorithm with improved accuracy close to singularities, while keeping the maximum possible accuracy at smooth zones. It is known that WENO algorithm does not attain the maximum possible accuracy close to discontinuities when there is more than one smooth sub-stencil. Although improving the accuracy of classical WENO algorithm close to discontinuities, we find that the technique presented in \cite{generalizacion} does not attain the maximum possible accuracy at some intervals when $r$ grows. In this article we solve this problem using a simple strategy that has to do with the design of the smoothness indicators of high order and we proof in general the accuracy of the new algorithm. We also give explicit formulas for all the weights used in the algorithm Finally, we particularize the proofs for low values of $r$ that are the most used in practice.

Let's now introduce how the WENO algorithm works. Let be $X$ a uniform partition of the interval
$[a, b]$ in $J$ subintervals, 
$$X=\{x_i\}^{J}_{i=0}, \quad x_i=a+i\cdot h, \quad h=\frac{b-a}{J}.$$
We will consider the point value discretization of the piecewise smooth function $f$ at the nodes $x_i$,
\begin{equation}\label{pv}
f_i=f(x_i),\, i=0,\hdots,J, \quad f=\left\{f_i\right\}_{i=0}^{J},
\end{equation}
and we will suppose that discontinuities are located far enough from each other, meaning that in a stencil we can only find one discontinuity.

In order to interpolate in the interval $(x_{i-1}, x_i)$, WENO-2r algorithm uses the stencil $\{x_{i-r},\cdots, x_{i+r-1}\}$, that is composed of $2r$ nodes.
Using the previous stencil, we can construct the convex combination,
\begin{equation}\label{convex_comb}
\mathcal{I}\left(x;f\right)=\sum_{k=0}^{r-1} \omega_{k}^r p_{k}^r(x),
\end{equation}
with the positive weights $\omega_{k}^r\ge 0, \, k=0,\cdots, r-1$ and assuring that
$\sum_{k=0}^{r-1}\omega_{k}^r=1$. In the expression (\ref{convex_comb}), the $r^{th}$ degree interpolation polynomials are denoted by $p_{k}^r(x)$.
We construct the interpolation at the mid point of the interval $(x_{i-1}, x_i)$, that will be denoted as $x_{i-\frac{1}{2}}$,
\begin{equation}\label{PO}
\mathcal{I}(x_{i-\frac{1}{2}};f)=\sum_{k=0}^{r-1}
\omega_{k}^r p_{k}^r(x_{i-\frac{1}{2}}).
\end{equation}
The values of the $\omega_{k}^r$ are forced to be those that allow to obtain order of accuracy $2r$ at $x_{i-\frac{1}{2}}$ when the stencil is smooth. When interpolating the discretized function $\{f(x_i)\}_{i=0}^{2r}$, the objective is to obtain an interpolation polynomial that satisfies,
\begin{equation}\nonumber
p_{0}^{2r-1}(x_{i-\frac{1}{2}})
=f(x_{i-{\frac{1}{2}}})+O(h^{2r}),
\end{equation}
based on the big stencil $\{x_{i-r},\cdots, x_{i+r-1}\}$, through the convex combination of the $r$ consecutive interpolation polynomials of order $r$,
\begin{equation}\nonumber
p_{k}^{r}(x_{i-\frac{1}{2}})=f(x_{i-{\frac{1}{2}}})
+O(h^{r+1}).
\end{equation}
We start reviewing classical WENO algorithm and its properties.

The classical WENO-2r interpolator, $\bar{\mathcal{I}}\left(x;f\right)$, imposes that the optimal weights are $\bar{C}_{k}^r\ge 0$, with $k=0,\hdots,r-1$, and
$\sum_{k=0}^{r-1} \bar{C}_{k}^r=1$, such that,
\begin{equation}\label{suma_pond}
p_{0}^{2r-1}\left(x_{i-\frac{1}{2}}\right)
=\sum_{k=0}^{r-1}\bar{C}_{k}^r p_{k}^r\left(x_{i-\frac{1}{2}}\right).
\end{equation}
A formula for the optimal weights is obtained in \cite{AMB},
\begin{equation}\label{opt_w}
\bar{C}_{k}^r=\frac{1}{2^{2r-1}}
\binom{2r}{2k+1}
,  \quad k=0,\cdots, r-1.
\end{equation}
The weights $\bar{\omega}_{k}^r$ are designed \cite{Liu} in order to satisfy at smooth zones that,
\begin{equation}\label{omega}
\bar{\omega}_{k}^r=\bar{C}_{k}^r+O(h^\kappa),\quad k=0,\cdots, r-1,
\end{equation}
with $\kappa\le r-1$, assuring that the interpolation in (\ref{PO}) attains order of accuracy $2r$ when $\kappa=r-1$, and
\begin{equation}\label{omega2}
f(x_{i-\frac{1}{2}})-\bar{\mathcal{I}}\left(x_{i-\frac{1}{2}};f\right)=O(h^{r+\kappa+1}),
\end{equation}
that matches the accuracy attained by the interpolation polynomial $p_{0}^{2r-1}(x)$ of $2r$ points. In \cite{Liu, JiangShu} the authors propose the following expressions for the nonlinear weights,
\begin{equation}\label{pesos}
\bar{\omega}_{k}^r=\frac{\bar{\alpha}_{k}^r}{\sum_{j=0}^{r-1}\bar{\alpha}_{j}^r},\quad  \textrm{ where } \bar{\alpha}_{k}^r=\frac{\bar{C}_{k}^r}{(\epsilon+\bar{I}_k^r)^t}, \quad k=0,\cdots, r-1,
\end{equation}
with $\sum_{k=0}^{r-1}\bar{\omega}_k^r=1$. In the previous expression, the parameter $t$ is an integer that assures maximum order of accuracy close to the discontinuities. The parameter $\epsilon>0$ is introduced to avoid divisions by zero and is usually forced to take the size of the smoothness indicators at smooth zones. In our numerical tests, we will set it to $\epsilon=10^{-16}$.  The values $\bar{I}_k^r$ are called {\it smoothness indicators} for $f(x)$ on each sub-stencil of $r$ points. The expression for the $\bar{I}_k^r$ initially given in \cite{AMB} is,
\begin{equation}\label{si_abm}
\bar{I}_k^r=\sum_{l=1}^{r-1} h^{2l-1}\int_{x_{i-1}}^{x_{i}}\left(\frac{d^l}{dx^l}p^r_{k}(x)\right)^2 dx.
\end{equation}
 In this paper, we generalize and improve the algorithms presented in \cite{generalizacion, WENO_nuevo} achieving maximum order of accuracy in the intervals close to the discontinuities for any value of $r$. We introduce the notation and review the previous results in Section \ref{resultadosprevios}. In order to design the new algorithm, the optimal and nonlinear weights are presented in Section \ref{newweno}. Afterwards, new smoothness indicators and its properties are proved in Section \ref{smoothnessindicators}. In Section \ref{analisisaccuracy}, we analyze the accuracy of the new method and finally we perform some experiments comparing the new method with the classical WENO and the new method in Section \ref{numexp}. 

\section{Review of previous results: The cases $r=3$ and $r=4$}\label{resultadosprevios}
In \cite{generalizacion, WENO_nuevo} it was presented a new WENO-2r algorithm that improves the resolution of classical WENO algorithms close to discontinuities. For WENO-6, i.e. r=3, the pattern of accuracy obtained with the new algorithm was $\cdots,O(h^6), O(h^5),O(h^4), O(1), O(h^4), O(h^5), O(h^6), \cdots$, that is the optimal accuracy that we can expect to obtain in the presence of a singularity (the one obtained by the classical WENO is typically $\cdots,O(h^6), O(h^4),O(h^4), O(1), O(h^4),\\ O(h^4), O(h^6), \cdots$). For higher orders, i.e. $r=4, 5, \cdots$, and despite of the fact that the new algorithm obtains a better accuracy than the classical WENO algorithm close to discontinuities, the theoretical pattern of accuracy obtained was not optimal, being for $r=4$, $\cdots, O(h^8), O(h^7), O(h^5), O(h^5), O(1),O(h^5), O(h^5), O(h^7), O(h^8) \cdots$, or for $r=5$, $\cdots, O(h^{10}), O(h^9), O(h^6), O(h^6),O(h^6),  O(1),  O(h^6), O(h^6), O(h^6), O(h^9), O(h^{10}) \cdots$, and so on.

The construction proposed in \cite{generalizacion} was based on a progressive construction of the building polynomials of WENO algorithm. This construction was based upon the observation that Lagrange interpolating polynomials of high order can be constructed from polynomials of lower order using a dyadic architecture. For example, for $r=3$ (stencil of 6 points) we use the polynomials of degree 3, $p^3_0(x), p^3_{1}(x)$ and $p^3_{2}(x)$ to write polynomials of degree 4. For a stencil of 6 points there exist two different Lagrange interpolating polynomials of degree 4 (stencil of 5 points), and we will note them as $p_0^4(x), p_1^4(x)$. There also exists one polynomial of degree 5 (stencil of 6 points), that we will note as $p_0^5(x)$. Now, it is clear that we can proceed to construct the polynomials of degree 4 using the polynomials of degree 3 as building blocks, we use the following notation:
\begin{equation}\label{w_opt_5_1}
\begin{aligned}
p^4_{0}(x_{i-1/2})&=C^3_{0,0}p^3_{0}(x_{i-1/2})+C^3_{0,1} p^3_{1}(x_{i-1/2})=\frac{3}{8} p^3_{0}(x_{i-1/2})+\frac{5}{8} p^3_{1}(x_{i-1/2}),\\
p^4_{1}(x_{i-1/2})&=C^3_{1,1}p^3_{1}(x_{i-1/2})+C^3_{1,2} p^3_{2}(x_{i-1/2})=\frac{5}{8} p^3_{1}(x_{i-1/2})+\frac{3}{8}  p^3_{2}(x_{i-1/2}).
\end{aligned}
\end{equation}
And use the polynomials of degree 4 as building blocks to construct the polynomial of degree 5,
\begin{equation}\label{w_opt_5_2}
\begin{aligned}
p^5_{0}(x_{i-1/2})&= C^4_{0,0}p^4_{0}(x_{i-1/2})+C^4_{0,1}p^4_{1}(x_{i-1/2})=\frac{1}{2} p^4_{0}(x_{i-1/2})+\frac{1}{2}p^4_{1}(x_{i-1/2}).
\end{aligned}
\end{equation}

Once we have reached this point, in \cite{generalizacion} we proposed to use the vectors of optimal weights ${\bf C_{0}^4, C_{1}^4}$ in the classical WENO algorithm. These vectors have as coordinates the weights in (\ref{w_opt_5_1}),
\begin{equation}\label{nl_op_w1}
\begin{aligned}
{\bf C_{0}^4}=\left(C^3_{0,0},C^3_{0,1}, 0\right)=\left(\frac{3}{8}, \frac{5}{8}, 0\right),\\
{\bf C_{1}^4}=\left(0, C^3_{1,1}, C^3_{1,2}\right)=\left(0, \frac{5}{8},\frac{3}{8}\right).
\end{aligned}
\end{equation}
The stencil that we use in this case is composed of data at the positions $\{x_{i-3}, x_{i-2}, x_{i-1}, x_{i}, x_{i+1}, x_{i+2}\}$. It is clear that it is convenient to use ${\bf C_{0}^4}$ when there is a discontinuity placed in $(x_{i+1}, x_{i+2})$ and ${\bf C_{1}^4}$ if it is placed in the interval $(x_{i-3}, x_{i-2})$. The objective is to obtain the weights in (\ref{opt_w}) for $r=3$ if the stencil does not contain any discontinuities, so that maximum accuracy is attained everywhere. In \cite{generalizacion} we proposed the weighted average of the vectors in (\ref{nl_op_w1}),
\begin{eqnarray}\label{nl_op_w22}
C^4_{0,0}{\bf C_{0}^4}+C^4_{0,1}{\bf C_{1}^4}=\frac{1}{2}{\bf C_{0}^4}+\frac{1}{2}{\bf C_{1}^4}=\frac{1}{2}\left(\frac{3}{8}, \frac{5}{8}, 0\right)+\frac{1}{2}\left(0, \frac{5}{8}, \frac{3}{8}\right)=\left(\frac{3}{16}, \frac{10}{16},\frac{3}{16}\right)=\left(\bar{C}_0^3, \bar{C}_1^3,  \bar{C}_2^3\right)={\bf \bar{C}^{3}}.
\end{eqnarray}
The reader can observe how the construction reminds us a WENO algorithm computed in several steps in order to grow once at a time the accuracy of the final interpolant. The obvious pace, leads us to define nonlinear weights for replacing the constant weights in (\ref{nl_op_w22}). We will represent the weights by
\begin{equation}\label{pesos31}
\begin{aligned}
\tilde{\omega}^4_{0,0}=\frac{\tilde{\alpha}_{0,0}^4}{\tilde{\alpha}_{0,0}^4+\tilde{\alpha}_{0,1}^4},\quad
\tilde{\omega}^4_{0,1}=\frac{\tilde{\alpha}_{0,1}^4}{\tilde{\alpha}_{0,0}^4+\tilde{\alpha}_{0,1}^4},
\end{aligned}
\end{equation}
with,
\begin{equation}\label{pesos2}
\begin{aligned}
\tilde{\alpha}_{0,0}^4=\frac{C^4_{0,0}}{(\epsilon+\tilde{I}_{0,0}^4)^t}=\frac{1/2}{(\epsilon+\tilde{I}_{0,0}^4)^t},\quad
\tilde{\alpha}_{0,1}^4=\frac{C^4_{0,1}}{(\epsilon+\tilde{I}_{0,1}^4)^t}=\frac{1/2}{(\epsilon+\tilde{I}_{0,1}^4)^t}.
\end{aligned}
\end{equation}
The result of this process is an expression for the adapted optimal weights of the classical WENO algorithm, that assure optimal accuracy for $r=3$, and that replace the classical constant optimal weights $\bar{C}_k^r$ in (\ref{pesos}),
\begin{eqnarray}\label{pesos3}
{\bf \tilde{C}^{3}}=(\tilde{C}_0^3,\tilde{C}_1^3,\tilde{C}_2^3)=\tilde{\omega}_{0,0}^4{\bf C_0^4}+\tilde{\omega}^4_{0,1}{\bf C_1^4}.
\end{eqnarray}
The smoothness indicators $\tilde{I}^4_{0,k_1}$, $k_1=0,1$ in (\ref{pesos2}) will be defined in Section \ref{smoothnessindicators} based on those introduced in \cite{WENO_nuevo,smooth}, that work well for detecting kinks and jumps in the function if the data is discretized in the point values (\ref{pv}). Thus, finally, we apply WENO with the new nonlinear weights, i.e., we calculate:
\begin{equation*}
\tilde{\mathcal{I}}(x_{i-\frac{1}{2}};f)=\sum_{k=0}^2
\tilde{\omega}_{k}^3 p_{k}^3(x_{i-\frac{1}{2}}).
\end{equation*}
with
\begin{equation}\label{pesos31finales}
\begin{aligned}
\tilde{\omega}^3_{k}=\frac{\tilde{\alpha}_{k}^3}{\sum_{j=0}^2 \tilde{\alpha}_j},\quad \text{and} \quad \tilde{\alpha}_{k}^3=\frac{\tilde{C}^3_{k}}{(\epsilon+\tilde{I}_{k}^3)^t}, \quad k=0,1,2,
\end{aligned}
\end{equation}
being $\tilde{I}_{k}^3$, $k=0,1,2$, the smoothness indicators proposed in \cite{smooth}.

For $r=4$ and higher values of $r$, it is possible to follow similar steps. In order to design the new WENO-8 algorithm we start with a stencil of 8 points composed of data placed at the positions $\{x_{i-4}, x_{i-3}, x_{i-2}, x_{i-1}, x_{i}, x_{i+1}, x_{i+2}, x_{i+3}\}$. In this case there exist four polynomials of degree four (stencil of five points), that will be noted by $p_0^4(x), p_1^4(x), p_2^4(x), \\p_3^4(x)$, three of degree five (stencil of six points), that we will denote by $p_0^5(x), p_1^5(x), p_2^5(x)$, two of degree six (stencil of seven points), denoted as $p_0^6(x), p_1^6(x)$ and one of degree seven (stencil of eight points), denoted as $p_0^7(x)$. The process is similar as before: we try to obtain nonlinear optimal weights that assure the optimal accuracy that the data of the stencil provides. As we did for $r=3$, we write the polynomials of degree five using the polynomials of degree four as building blocks,
\begin{equation}\label{pol5_8}
\begin{aligned}
p^5_{0}(x_{i-1/2})&=C^4_{0,0} p^4_{0}(x_{i-1/2})+C^4_{0,1} p^4_{1}(x_{i-1/2})=\frac{3}{10} p^4_{0}(x_{i-1/2})+\frac{7}{10} p^4_{1}(x_{i-1/2}),\\
p^5_{1}(x_{i-1/2})&=C^4_{1,1} p^4_{1}(x_{i-1/2})+C^4_{1,2}  p^4_{2}(x_{i-1/2})=\frac{1}{2} p^4_{1}(x_{i-1/2})+\frac{1}{2}  p^4_{2}(x_{i-1/2}),\\
p^5_{2}(x_{i-1/2})&=C^4_{2,2} p^4_{2}(x_{i-1/2})+C^4_{2,3}  p^4_{3}(x_{i-1/2})=\frac{7}{10} p^4_{2}(x_{i-1/2})+\frac{3}{10}  p^4_{3}(x_{i-1/2}).
\end{aligned}
\end{equation}
We write the polynomials of degree six using the polynomials of degree five as building blocks,
\begin{equation}\label{pol6_8}
\begin{aligned}
p^6_{0}(x_{i-1/2})&=C^5_{0,0} p^5_{0}(x_{i-1/2})+C^5_{0,1} p^5_{1}(x_{i-1/2})=\frac{5}{12} p^5_{0}(x_{i-1/2})+\frac{7}{12} p^5_{1}(x_{i-1/2}),\\
p^6_{1}(x_{i-1/2})&=C^5_{1,1} p^5_{1}(x_{i-1/2})+C^5_{1,2}  p^5_{2}(x_{i-1/2})=\frac{7}{12} p^5_{1}(x_{i-1/2})+\frac{5}{12}  p^5_{2}(x_{i-1/2}).
\end{aligned}
\end{equation}
In this last step, the polynomial of degree seven is written in terms of the two polynomials of degree six,
\begin{equation}\label{pol7_8}
\begin{aligned}
p^7_{0}(x_{i-1/2})&=C^6_{0,0} p^6_{0}(x_{i-1/2})+C^6_{0,1} p^6_{1}(x_{i-1/2})=\frac{1}{2} p^6_{0}(x_{i-1/2})+\frac{1}{2} p^6_{1}(x_{i-1/2}).\\
\end{aligned}
\end{equation}
In a similar fashion as we did for $r=3$, we construct the vectors of weights ${\bf C_{0}^5, C_{1}^5, C_{2}^5}$, that will have as coordinates the constant weights in (\ref{pol5_8}),
\begin{eqnarray}\label{nl_op_w}
\scalemath{0.9}{
{\bf C_{0}^5}=\left(C^4_{0,0} , C^4_{0,1} , 0, 0\right)=\left(\frac{3}{10}, \frac{7}{10}, 0, 0\right),\quad {\bf C_{1}^5}=\left(0,C^4_{1,1} , C^4_{1,2} , 0\right)=\left(0, \frac{1}{2},\frac{1}{2}, 0\right),\quad {\bf C_{2}^5}=\left(0,0,C^4_{2,2} , C^4_{2,3} \right)=\left(0, 0, \frac{7}{10},\frac{3}{10}\right).
}
\end{eqnarray}
If we multiply the previous vectors by the constant weights calculated in (\ref{pol6_8}) and (\ref{pol7_8}) we directly obtain the weights that assure optimal accuracy at smooth zones, as they are the constant optimal weights used by the classical WENO-8 algorithm, i.e.
\begin{equation}\label{nl_op_w2}
\begin{split}
C^6_{0,0}\left(C^5_{0,0}{\bf C_{0}^5}+C^5_{0,1}{\bf C_{1}^5}\right)+C^6_{0,1}\left(C^5_{1,1}{\bf C_{1}^5}+C^5_{1,2}{\bf C_{2}^5}\right)&=\frac{1}{2}\left(\frac{5}{12}{\bf C_{0}^5}+\frac{7}{12}{\bf C_{1}^5}\right)+\frac{1}{2}\left(\frac{7}{12}{\bf C_{1}^5}+\frac{5}{12}{\bf C_{2}^5}\right)\\
&=\left(\frac{1}{16}, \frac{7}{16},\frac{7}{16},\frac{1}{16}\right)\\
&=\left(\bar{C}_0^4, \bar{C}_1^4,  \bar{C}_2^4, \bar{C}_3^4\right)={\bf \bar{C}^{4}}\\
\end{split}
\end{equation}
It is clear that we can replace the constant weights in (\ref{nl_op_w2}) by nonlinear weights. As before, we will represent by $\tilde{\omega}_{k,k_1}^l$, $5\leq l\leq 6$ and $0 \leq k\leq 6-l$, $k_1=k, k+1$; the nonlinear weights are,
\begin{equation}\label{pesos4}
\scalemath{0.9}{
\begin{aligned}
\tilde{\omega}^5_{0,0}=\frac{\tilde{\alpha}_{0,0}^5}{\tilde{\alpha}_{0,0}^5+\tilde{\alpha}_{0,1}^5},\quad
\tilde{\omega}^5_{0,1}=\frac{\tilde{\alpha}_{0,1}^5}{\tilde{\alpha}_{0,0}^5+\tilde{\alpha}_{0,1}^5},\quad
\tilde{\omega}^5_{1,1}=\frac{\tilde{\alpha}_{1,1}^5}{\tilde{\alpha}_{1,1}^5+\tilde{\alpha}_{1,2}^5},\quad
\tilde{\omega}^5_{1,2}=\frac{\tilde{\alpha}_{1,2}^5}{\tilde{\alpha}_{1,1}^5+\tilde{\alpha}_{1,2}^5},\quad
\tilde{\omega}^6_{0,0}=\frac{\tilde{\alpha}_{0,0}^6}{\tilde{\alpha}_{0,0}^6+\tilde{\alpha}_{0,1}^6},\quad
\tilde{\omega}^6_{0,1}=\frac{\tilde{\alpha}_{0,1}^6}{\tilde{\alpha}_{0,0}^6+\tilde{\alpha}_{0,1}^6},
\end{aligned}
}
\end{equation}
with,
\begin{equation}\label{pesos4_2rr}
\scalemath{0.9}{
\begin{aligned}
\tilde{\alpha}_{0,0}^5=\frac{C^5_{0,0}}{(\epsilon+\tilde{I}_{0,0}^5)^t},\quad
\tilde{\alpha}_{0,1}^5=\frac{C^5_{0,1}}{(\epsilon+\tilde{I}_{0,1}^5)^t},\quad
\tilde{\alpha}_{1,1}^5=\frac{C^5_{1,1}}{(\epsilon+\tilde{I}_{1,1}^5)^t},\quad
\tilde{\alpha}_{1,2}^5=\frac{C^5_{1,2}}{(\epsilon+\tilde{I}_{1,2}^5)^t},\quad
\tilde{\alpha}_{0,0}^6=\frac{C^6_{0,0}}{(\epsilon+\tilde{I}_{0,0}^6)^t},\quad
\tilde{\alpha}_{0,1}^6=\frac{C^6_{0,1}}{(\epsilon+\tilde{I}_{0,1}^6)^t}, \end{aligned}
}
\end{equation}
that replacing the values of the $C_{k,k_1}^l$ given in (\ref{nl_op_w2}) result in,
\begin{equation}\label{pesos4_2}
\scalemath{0.9}{
\begin{aligned}
\tilde{\alpha}_{0,0}^5=\frac{5/12}{(\epsilon+\tilde{I}_{0,0}^5)^t},\quad
\tilde{\alpha}_{0,1}^5=\frac{7/12}{(\epsilon+\tilde{I}_{0,1}^5)^t},\quad
\tilde{\alpha}_{1,1}^5=\frac{7/12}{(\epsilon+\tilde{I}_{1,1}^5)^t},\quad
\tilde{\alpha}_{1,2}^5=\frac{5/12}{(\epsilon+\tilde{I}_{1,2}^5)^t},\quad
\tilde{\alpha}_{0,0}^6=\frac{1/2}{(\epsilon+\tilde{I}_{0,0}^6)^t},\quad
\tilde{\alpha}_{0,1}^6=\frac{1/2}{(\epsilon+\tilde{I}_{0,1}^6)^t}.
\end{aligned}
}
\end{equation}
Just as we did before, the adapted optimal weights are obtained replacing the fixed weights in (\ref{nl_op_w2}) by the nonlinear weights in (\ref{pesos4}),
\begin{eqnarray}\label{pesos4_3}
{\bf \tilde{C}^{4}}=(\tilde{C}_0^4,\tilde{C}_1^4,\tilde{C}_2^4,\tilde{C}_3^4)=\tilde{\omega}_{0,0}^6\left(\tilde{\omega}_{0,0}^5{\bf C_0^5}+\tilde{\omega}_{0,1}^5{\bf C_1^5}\right)+\tilde{\omega}_{0,1}^6\left(\tilde{\omega}_{1,1}^5{\bf C_1^5}+\tilde{\omega}_{1,2}^5{\bf C_2^5}\right).
\end{eqnarray}
The previous expression provides the nonlinear optimal weights $\tilde{C}_k^r$ that are used to replace the optimal weights $\bar{C}_{k}^r$ in (\ref{pesos}) of the classical WENO algorithm. The smoothness indicators that appear in (\ref{pesos4_3}) will be defined in Section \ref{smoothnessindicators}. 
Analogously to the case $r=3$, we apply the classical WENO algorithm with the new nonlinear weights:
\begin{equation*}
\tilde{\mathcal{I}}(x_{i-\frac{1}{2}};f)=\sum_{k=0}^{3}
\tilde{\omega}_{k}^4 p_{k}^4(x_{i-\frac{1}{2}}).
\end{equation*}
with
\begin{equation}\label{pesos41finales}
\begin{aligned}
\tilde{\omega}^4_{k}=\frac{\tilde{\alpha}_{k}^4}{\sum_{j=0}^3 \tilde{\alpha}_j},\quad \text{with} \quad \tilde{\alpha}_{k}^4=\frac{\tilde{C}^4_{k}}{(\epsilon+\tilde{I}_{k}^4)^t}, \quad k=0,1,2,3,
\end{aligned}
\end{equation}
being $\tilde{I}_{k}^4$, $k=0,1,2,3$, the smoothness indicators defined in \cite{smooth,WENO_nuevo}.

We can generalize this process for any $r$. The difficulty lies in constructing smoothness indicators for each level that are able \emph{to watch} the discontinuities correctly. In the next section, we present the generalization of the algorithm for order $2r$ and explicit expressions for all the weights.

\section{General explicit expressions for new WENO-$2r$ weights}\label{newweno}
Following the strategy described in previous section, we can construct a WENO-$2r$ algorithm for any $r$ with optimal accuracy.
In order to extend the results to any $r$ we will use the following lemma.
\begin{lemma}\label{lemaauxiliar1}
Let be $r\leq l \leq 2r-2$ and $0\leq k\leq (2r-2)-l$, if we denote as $C^l_{k,k}$ and $C^{l}_{k,k+1}$ the values which satisfy:
\begin{equation}
p^{l+1}_k(x_{i-1/2})=C^{l}_{k,k}p_k^{l}(x_{i-1/2})+C^{l}_{k,k+1}p_{k+1}^{l}(x_{i-1/2}),
\end{equation}
then
\begin{equation}\label{pesostodos}
C^l_{k,k}=\frac{2(l-r+k+1)+1}{2(l+1)}, \quad C^l_{k,k+1}=1-C^l_{k,k}=\frac{2(r-k)-1}{2(l+1)}.
\end{equation}
\end{lemma}
\begin{proof}
Let be $r\leq l \leq 2r-2$ and $0\leq k\leq (2r-2)-l$, the stencils used to obtain the interpolators $p_k^{l+1}$, $p_k^{l}$ and $p_{k+1}^{l}$ are
$\{x_{i-r+k},\dots,x_{i-r+k+l+1}\}$, $\{x_{i-r+k},\dots,x_{i-r+k+l}\}$ and $\{x_{i-r+k+1},\dots,x_{i-r+k+l+1}\}$ respectively, then using Aitken's interpolation process \cite{Aitken}, from $x_i=a+ih$, we get:
\begin{equation*}
\begin{split}
p^{l+1}_k(x_{i-\frac12})&= \frac{x_{i-r+k+l+1}-x_{i-1/2}}{x_{i-r+k+l+1}-x_{i-r+k}}p_k^{l}(x_{i-1/2})-\frac{x_{i-r+k}-x_{i-1/2}}{x_{i-r+k+l+1}-x_{i-r+k}}p_{k+1}^{l}(x_{i-1/2})\\
&= \frac{a+(i-r+k+l+1)h-(a+(i-1/2)h)}{a+(i-r+k+l+1)h-(a+(i-r+k)h)}p_k^{l}(x_{i-\frac12})-\frac{a+(i-r+k)h-(a+(i-1/2)h)}{a+(i-r+k+l+1)h-(a+(i-r+k)h)}p_{k+1}^{l}(x_{i-\frac12})\\
&= \frac{2(l-r+k+1)+1}{2(l+1)}p_k^{l}(x_{i-\frac12})+\frac{2(r-k)-1}{2(l+1)}p_{k+1}^{l}(x_{i-\frac12})\\
&= C^{l}_{k,k}p_k^{l}(x_{i-1/2})+C^{l}_{k,k+1}p_{k+1}^{l}(x_{i-1/2}).\\
\end{split}
\end{equation*}
\end{proof}

We make the same construction for any $r$, we start writing the polynomials of degree $2r-1$ as combination of polynomials of degree $2r-2$,
 $$p^{2r-1}_0(x_{i-1/2})=C^{2r-2}_{0,0}p_0^{2r-2}(x_{i-1/2})+C^{2r-2}_{0,1}p_{1}^{2r-2}(x_{i-1/2})$$
 and repeat the process for degree $2r-2$,
\begin{equation*}
\begin{split}
&p^{2r-2}_0(x_{i-1/2})=C^{2r-3}_{0,0}p_0^{2r-3}(x_{i-1/2})+C^{2r-3}_{0,1}p_{1}^{2r-3}(x_{i-1/2}),\\
&p^{2r-2}_1(x_{i-1/2})=C^{2r-3}_{1,1}p_1^{2r-3}(x_{i-1/2})+C^{2r-3}_{1,2}p_{2}^{2r-3}(x_{i-1/2}),
\end{split}
\end{equation*}
for degree $2r-3$,
\begin{equation*}
\begin{split}
&p^{2r-3}_0(x_{i-1/2})=C^{2r-4}_{0,0}p_0^{2r-4}(x_{i-1/2})+C^{2r-4}_{0,1}p_{1}^{2r-4}(x_{i-1/2}),\\
&p^{2r-3}_1(x_{i-1/2})=C^{2r-4}_{1,1}p_1^{2r-4}(x_{i-1/2})+C^{2r-4}_{1,2}p_{2}^{2r-4}(x_{i-1/2}),\\
&p^{2r-3}_2(x_{i-1/2})=C^{2r-4}_{2,2}p_2^{2r-4}(x_{i-1/2})+C^{2r-4}_{2,3}p_{3}^{2r-4}(x_{i-1/2}),
\end{split}
\end{equation*}
and so on, until we reach the polynomials of degree $r+1$
 $$p^{r+1}_l(x_{i-1/2})=C^{r}_{l,l}p_l^{r}(x_{i-1/2})+C^{r}_{l,l+1}p_{l+1}^{r}(x_{i-1/2}), \quad l=0,\hdots,r-2.$$
 Thus, if we combine these equations we obtain,
\begin{equation}\label{equationimpo}
\begin{split}
p^{2r-1}_0(x_{i-\frac 12})&=\sum_{j_0=0}^1C^{2r-2}_{0,j_0}p_{j_0}^{2r-2}(x_{i-1/2}) \\
&=\sum_{j_0=0}^1C^{2r-2}_{0,j_0}\left(\sum_{j_1=j_0}^{j_0+1} C^{2r-3}_{j_0,j_1} p_{j_1}^{2r-3}(x_{i-\frac12})\right)\\
&=\sum_{j_0=0}^1C^{2r-2}_{0,j_0}\left(\sum_{j_1=j_0}^{j_0+1} C^{2r-3}_{j_0,j_1} \left(\sum_{j_2=j_1}^{j_1+1} C^{2r-4}_{j_1,j_2}p_{j_2}^{2r-4}(x_{i-\frac12})\right)\right)\\
&=\sum_{j_0=0}^1 C^{2r-2}_{0,j_0}\left(\sum_{j_1=j_0}^{j_0+1}C_{j_0,j_1}^{2r-3}\left(\sum_{j_2=j_1}^{j_1+1}C_{j_1,j_2}^{2r-4}
\left(\dots\left(\sum_{j_{r-2}=j_{r-3}}^{j_{r-3}+1} C^{r+1}_{j_{r-3},j_{r-2}}\left(\sum_{j_{r-1}=j_{r-2}}^{j_{r-2}+1} C^{r}_{j_{r-2},j_{r-1}}p^r_{j_{r-1}}(x_{i-\frac12})\right)\right)\dots\right)\right)\right)
\end{split}
\end{equation}
In Figure \ref{figuraesquema} we show this process for any $r$. We can observe that the diagram shows a tree structure where from each node $C^{l}_{k,k_1}$, with $k_1=k$ or $k_1=k+1$,
it is obtained two subnodes of the form $C^{l-1}_{k_1,k_2}$ with $k_2=k_1$ and $k_2=k_1+1$, which allows an easy construction. The explicit values, $C^{l}_{k,k_1}$ , are calculated using Eq. \eqref{pesostodos} of Lemma \ref{lemaauxiliar1}.
\begin{figure}
\begin{center}
\begin{tikzpicture}
\draw [-] (-0.5,16) -- (0.5,18);
\draw [-] (-0.5,16) -- (0.5,14);
\draw [-] (-0.5,8) -- (0.5,10);
\draw [-] (-0.5,8) -- (0.5,6);

\draw [-] (1.5,18) -- (2.5,19);
\draw [-] (1.5,18) -- (2.5,17);
\draw [-] (1.5,14) -- (2.5,15);
\draw [-] (1.5,14) -- (2.5,13);
\draw [-] (1.5,10) -- (2.5,11);
\draw [-] (1.5,10) -- (2.5,9);
\draw [-] (1.5,6) -- (2.5,7);
\draw [-] (1.5,6) -- (2.5,5);

\draw [dashed] (3.5,19) -- ( 4.5,19.5);
\draw [dashed] (3.5,19) -- ( 4.5,18.5);
\draw [dashed] (3.5,17) -- ( 4.5,17.5);
\draw [dashed] (3.5,17) -- ( 4.5,16.5);
\draw [dashed] (3.5,15) -- ( 4.5,15.5);
\draw [dashed] (3.5,15) --  ( 4.5,14.5);
\draw [dashed] (3.5,13) --  ( 4.5,13.5);
\draw [dashed] (3.5,13) --  ( 4.5,12.5);
\draw [dashed] (3.5,11) -- ( 4.5,11.5);
\draw [dashed] (3.5,11) -- ( 4.5,10.5);
\draw [dashed] (3.5,9) -- ( 4.5,9.5);
\draw [dashed] (3.5,9) -- ( 4.5,8.5);
\draw [dashed] (3.5,7) -- ( 4.5,7.5);
\draw [dashed] (3.5,7) --  ( 4.5,6.5);
\draw [dashed] (3.5,5) --  ( 4.5,5.5);
\draw [dashed] (3.5,5) --  ( 4.5,4.5);

\draw [-] ( 6.5,21) -- ( 7.4,21.5);
\draw [-] ( 6.5,21) -- ( 7.4,20.5);
\draw [-] ( 6.5,19) -- ( 7.4,19.5);
\draw [-] ( 6.5,19) -- ( 7.4,18.5);
\draw [-] ( 6.5,17) -- ( 7.4,17.5);
\draw [-] ( 6.5,17) -- ( 7.4,16.5);
\draw [-] ( 6.5,15) -- ( 7.4,15.5);
\draw [-] ( 6.5,15) -- ( 7.4,14.5);
\draw [-] ( 6.5,13) -- ( 7.4,13.5);
\draw [-] ( 6.5,13) -- ( 7.4,12.5);
\draw [-] ( 6.5,11) -- ( 7.4,11.5);
\draw [-] ( 6.5,11) -- ( 7.4,10.5);
\draw [-] ( 6.5,9)  -- ( 7.4,9.5) ;
\draw [-] ( 6.5,9)  -- ( 7.4,8.5) ;
\draw [-] ( 6.7,5)  -- ( 6.9,5.5)   ;
\draw [-] ( 6.7,5)  -- ( 6.9,4.5) ;
\draw [-] ( 6.7,3)  -- ( 6.9,3.5);
\draw [-] ( 6.7,3)  -- ( 6.9,2.5);

\path node at ( -1,16) {$C^{2r-2}_{0,0}$}
node at ( -1,8) {$C^{2r-2}_{0,1}$}

node at ( 1,18) {$C^{2r-3}_{0,0}$}
node at ( 1,14) {$C^{2r-3}_{0,1}$}
node at ( 1,10) {$C^{2r-3}_{1,1}$}
node at ( 1,6) {$C^{2r-3}_{1,2}$}

node at ( 3,19) {$C^{2r-4}_{0,0}$}
node at ( 3,17) {$C^{2r-4}_{0,1}$}
node at ( 3,15) {$C^{2r-4}_{1,1}$}
node at ( 3,13) {$C^{2r-4}_{1,2}$}
node at ( 3,11) {$C^{2r-4}_{1,1}$}
node at ( 3,9) {$C^{2r-4}_{1,2}$}
node at ( 3,7) {$C^{2r-4}_{2,2}$}
node at ( 3,5) {$C^{2r-4}_{2,3}$}

node at ( 5,19.5) {$\hdots$}
node at ( 5,18.5) {$\hdots$}
node at ( 5,17.5) {$\hdots$}
node at ( 5,16.5) {$\hdots$}
node at ( 5,15.5) {$\hdots$}
node at ( 5,14.5) {$\hdots$}
node at ( 5,13.5) {$\hdots$}
node at ( 5,12.5) {$\hdots$}
node at ( 5,11.5) {$\hdots$}
node at ( 5,10.5) {$\hdots$}
node at ( 5,9.5) {$\hdots$}
node at ( 5,8.5) {$\hdots$}
node at ( 5,7.5) {$\hdots$}
node at ( 5,6.5) {$\hdots$}
node at ( 5,5.5) {$\hdots$}
node at ( 5,4.5) {$\hdots$}

node at ( 6,21) {$C^{r+1}_{0,0}$}
node at ( 6,19) {$C^{r+1}_{1,1}$}
node at ( 6,17) {$C^{r+1}_{1,2}$}
node at ( 6,15) {$C^{r+1}_{1,1}$}
node at ( 6,13) {$C^{r+1}_{1,2}$}
node at ( 6,11) {$C^{r+1}_{1,2}$}
node at ( 6,9) {$C^{r+1}_{2,3}$}
node at ( 6,7) {$\vdots$}
node at ( 6,5) {$C^{r+1}_{r-3,r-3}$}
node at ( 6,3) {$C^{r+1}_{r-3,r-2}$}

node at ( 8,21.5) {$C^{r}_{0,0}p^r_0$}
node at ( 8,20.5) {$C^{r}_{0,1}p^r_1$}
node at ( 8,19.5) {$C^{r}_{1,1}p^r_1$}
node at ( 8,18.5) {$C^{r}_{1,2}p^r_2$}
node at ( 8,17.5) {$C^{r}_{2,2}p^r_2$}
node at ( 8,16.5) {$C^{r}_{2,3}p^r_3$}
node at ( 8,15.5) {$C^{r}_{1,1}p^r_1$}
node at ( 8,14.5) {$C^{r}_{1,2}p^r_2$}
node at ( 8,13.5) {$C^{r}_{2,2}p^r_2$}
node at ( 8,12.5) {$C^{r}_{2,3}p^r_3$}
node at ( 8,11.5) {$C^{r}_{2,2}p^r_2$}
node at ( 8,10.5) {$C^{r}_{2,3}p^r_3$}
node at ( 8,9.5) {$C^{r}_{3,3}p^r_3$}
node at ( 8,8.5) {$C^{r}_{3,4}p^r_4$}
node at ( 8,7)   {$\vdots$}
node at ( 8,5.5) {$C^{r}_{r-3,r-3}p^r_{r-3}$}
node at ( 8,4.5) {$C^{r}_{r-3,r-2}p^r_{r-2}$}
node at ( 8,3.5) {$C^{r}_{r-2,r-2}p^r_{r-2}$}
node at ( 8,2.5) {$C^{r}_{r-2,r-1}p^r_{r-1}$};
\end{tikzpicture}
\end{center}
\caption{Diagram showing the structure of the optimal weights needed to obtain optimal order of accuracy.}
\label{figuraesquema}
\end{figure}
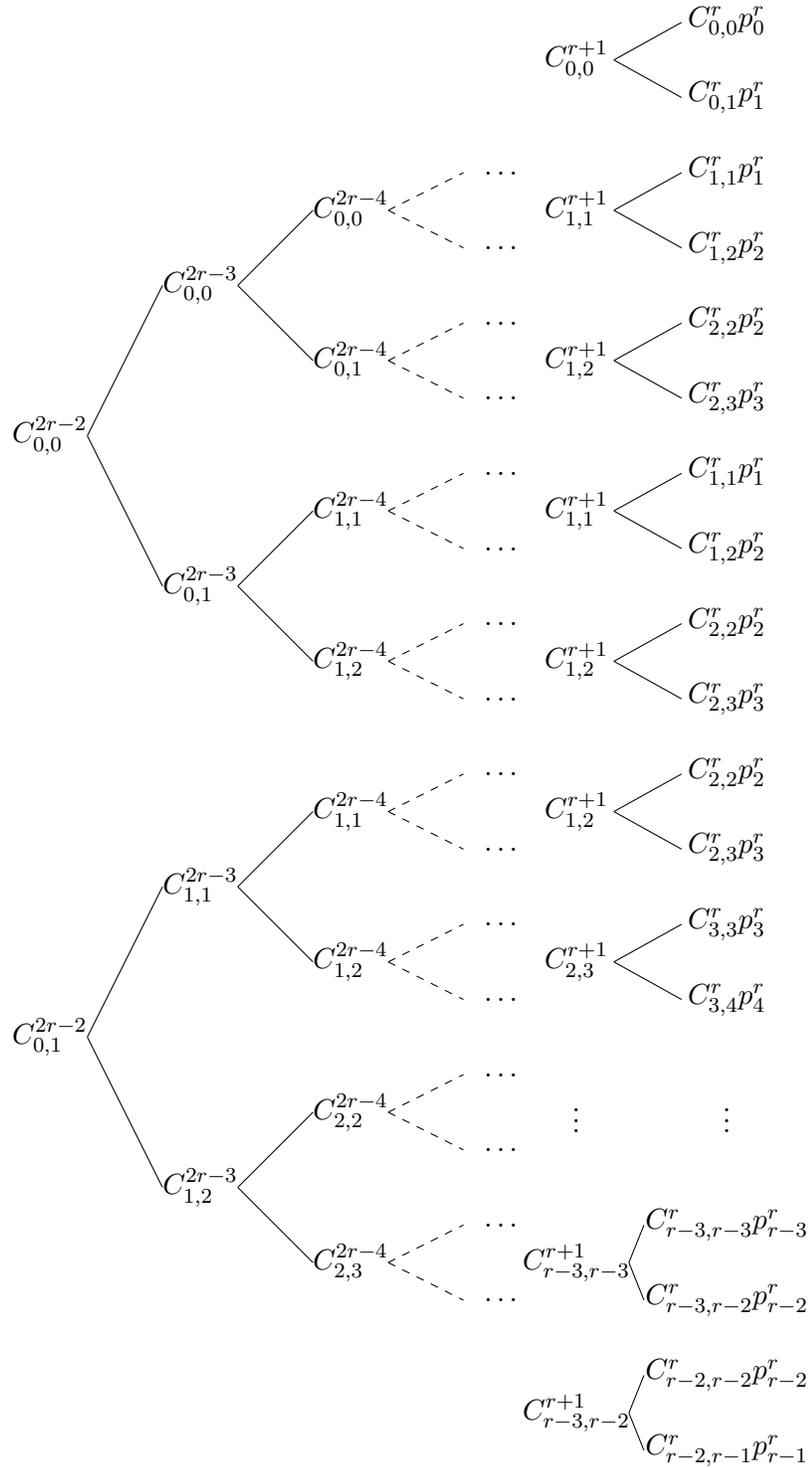
Therefore, it is easy to prove that the weights and the vector $\bf{C_k^{r+1}}$ with $0\leq k\leq r-2$, are
\begin{equation}\label{nl_op_w_r}
\begin{split}
&{\bf C_{0}^{r+1}}=\left(C_{0,0}^r, C_{0,1}^r, 0,0, \hdots,0\right)=\left(\frac{3}{2(r+1)}, \frac{2r-1}{2(r+1)},0, 0, \dots,0\right),\\
&{\bf C_{1}^{r+1}}=\left(0,C_{1,1}^r, C_{1,2}^r,0, \hdots, 0\right)=\left(0, \frac{5}{2(r+1)},\frac{2r-3}{2(r+1)},0,\dots,0\right),\\
 &\qquad \qquad\qquad\vdots\\
&{\bf C_{r-3}^{r+1}}=\left(0,\hdots,0,C_{r-3,r-3}^r, C_{r-3,r-2}^r, 0\right)=\left(0,\hdots,0,\frac{2r-3}{2(r+1)},\frac{5}{2(r+1)}, 0\right),\\
&{\bf C_{r-2}^{r+1}}=\left(0,\hdots,0, 0,C_{r-2,r-2}^r, C_{r-2,r-1}^r\right)=\left(0,\hdots,0, 0, \frac{2r-1}{2(r+1)},\frac{3}{2(r+1)}\right).\\
\end{split}
\end{equation}

Hence, by construction, these weights satisfy
\begin{equation}\label{eqsuperimp0}
\begin{split}
\sum_{j_0=0}^1 C^{2r-2}_{0,j_0}\left(\sum_{j_1=j_0}^{j_0+1}C_{j_0,j_1}^{2r-3}\left(\sum_{j_2=j_1}^{j_1+1}C_{j_1,j_2}^{2r-4}
\left(\dots\left(\sum_{j_{r-2}=j_{r-3}}^{j_{r-3}+1} C^{r+1}_{j_{r-3},j_{r-2}}{\bf C_{j_{r-2}}^{r+1}} \right)\dots\right)\right)\right)=(\bar{C}_0^r,\bar{C}_1^r,\hdots,\bar{C}_{r-2}^r,\bar{C}_{r-1}^r)={\bf \bar{C}^{r}}.
\end{split}
\end{equation}
In order to design the nonlinear weights, we substitute in Eq. \eqref{eqsuperimp0} the values $C_{k,k_1}^{l}$, for $l=r+1,\dots,2r-2$ and $0 \leq k\leq (2r-2)-l$, $k_1=k,k+1$ by
\begin{equation}\label{pesosr}
\begin{split}
&\tilde{\omega}^l_{k,k_1}=\frac{\tilde{\alpha}_{k,k_1}^l}{\tilde{\alpha}_{k,k}^l+\tilde{\alpha}_{k,k+1}^l},\quad k_1=k,\,\, k+1,\\
&\tilde{\alpha}_{k,k_1}^l=\frac{C_{k,k_1}^l}{(\epsilon+\tilde I^l_{k,k_1})^t}, \quad k_1=k,\,\,k+1,\\
\end{split}
\end{equation}
where $\tilde I^l_{k,k_1}$ are the smoothness indicators introduced in \cite{WENO_nuevo} and that will be analyzed in detail in Section \ref{generalsmoothness}. Thus, the nonlinear weights are defined as
\begin{equation}\label{eqsuperimp}
\begin{split}
{\bf \tilde{C}^{r}}=(\tilde{C}_0^r,\tilde{C}_1^r,\hdots,\tilde{C}_{r-2}^r,\tilde{C}_{r-1}^r)=\sum_{j_0=0}^1 \tilde{\omega}^{2r-2}_{0,j_0}\left(\sum_{j_1=j_0}^{j_0+1}\tilde{\omega}_{j_0,j_1}^{2r-3}\left(\sum_{j_2=j_1}^{j_1+1}\tilde{\omega}_{j_1,j_2}^{2r-4}
\left(\dots\left(\sum_{j_{r-2}=j_{r-3}}^{j_{r-3}+1} \tilde{\omega}^{r+1}_{j_{r-3},j_{r-2}}{\bf C_{j_{r-2}}^{r+1}} \right)\dots\right)\right)\right).
\end{split}
\end{equation}
Finally, for these weights, we apply the WENO algorithm described in Section \ref{introduction} using the nonlinear optimal weights in (\ref{eqsuperimp}), i.e., we calculate
\begin{equation*}
\tilde{\mathcal{I}}(x_{i-\frac{1}{2}};f)=\sum_{k=0}^{r}
\tilde{\omega}_{k}^r p_{k}^r(x_{i-\frac{1}{2}}).
\end{equation*}
with
\begin{equation}\label{pesosr1finales}
\begin{aligned}
\tilde{\omega}^r_{k}=\frac{\tilde{\alpha}_{k}^r}{\sum_{j=0}^{r-1} \tilde{\alpha}_j},\quad \text{with} \quad \tilde{\alpha}_{k}^r=\frac{\tilde{C}^r_{k}}{(\epsilon+\tilde{I}_{k}^r)^t}, \quad k=0,\hdots,r-1.
\end{aligned}
\end{equation}
being again the $\tilde{I}_{k}^r$, $k=0,\hdots,r$, the smoothness indicators defined in \cite{smooth,WENO_nuevo}.

In order to clarify this construction, we present the case $r=5$ because $r=3$ and $r=4$ have already been displayed in previous section.

\subsection{The case $r=5$}
In order to construct the nonlinear weights for $r=5$, firstly, we calculate the vectors $\bf{C_k^6}$, with $k=0,1,2,3$, thus
\begin{equation}\label{nl_op_w_5}
\begin{split}
&{\bf C_{0}^6}=\left(C_{0,0}^5, C_{0,1}^5, 0, 0,0\right)=\left(\frac{1}{4}, \frac{3}{4}, 0, 0,0\right),\quad {\bf C_{1}^6}=\left(0,C_{1,1}^5, C_{1,2}^5, 0, 0\right)=\left(0, \frac{5}{12},\frac{7}{12}, 0,0\right),\\
 &{\bf C_{2}^6}=\left(0,0,C_{2,2}^5, C_{2,3}^5, 0\right)=\left(0, 0, \frac{7}{12},\frac{5}{12},0\right),\quad {\bf C_{3}^6}=\left(0, 0,0,C_{3,3}^5, C_{3,4}^5\right)=\left(0, 0, 0,\frac{3}{4},\frac{1}{4}\right).
\end{split}
\end{equation}

Secondly, we obtain the expressions of the vectors:
\begin{equation}\label{pesos5}
\begin{aligned}
\tilde{\omega}^6_{0,0}&=\frac{\tilde{\alpha}_{0,0}^6}{\tilde{\alpha}_{0,0}^6+\tilde{\alpha}_{0,1}^6},\quad
\tilde{\omega}^6_{0,1}=\frac{\tilde{\alpha}_{0,1}^6}{\tilde{\alpha}_{0,0}^6+\tilde{\alpha}_{0,1}^6},\quad
\tilde{\omega}^6_{1,1}=\frac{\tilde{\alpha}_{1,1}^6}{\tilde{\alpha}_{1,1}^6+\tilde{\alpha}_{1,2}^6},\quad
\tilde{\omega}^6_{1,2}=\frac{\tilde{\alpha}_{1,2}^6}{\tilde{\alpha}_{1,1}^6+\tilde{\alpha}_{1,2}^6},\quad
\tilde{\omega}^6_{2,2}=\frac{\tilde{\alpha}_{2,2}^6}{\tilde{\alpha}_{2,2}^6+\tilde{\alpha}_{2,3}^6},\quad
\tilde{\omega}^6_{2,3}=\frac{\tilde{\alpha}_{2,3}^6}{\tilde{\alpha}_{2,2}^6+\tilde{\alpha}_{2,3}^6},\\
\tilde{\omega}^7_{0,0}&=\frac{\tilde{\alpha}_{0,0}^7}{\tilde{\alpha}_{0,0}^7+\tilde{\alpha}_{0,1}^7},\quad
\tilde{\omega}^7_{0,1}=\frac{\tilde{\alpha}_{0,1}^7}{\tilde{\alpha}_{0,0}^7+\tilde{\alpha}_{0,1}^7},\quad
\tilde{\omega}^7_{1,1}=\frac{\tilde{\alpha}_{1,1}^7}{\tilde{\alpha}_{1,1}^7+\tilde{\alpha}_{1,2}^7},\quad
\tilde{\omega}^7_{1,2}=\frac{\tilde{\alpha}_{1,2}^7}{\tilde{\alpha}_{1,1}^7+\tilde{\alpha}_{1,2}^7},\\
\tilde{\omega}^8_{0,0}&=\frac{\tilde{\alpha}_{0,0}^8}{\tilde{\alpha}_{0,0}^8+\tilde{\alpha}_{0,1}^8},\quad
\tilde{\omega}^8_{0,1}=\frac{\tilde{\alpha}_{0,1}^8}{\tilde{\alpha}_{0,0}^8+\tilde{\alpha}_{0,1}^8}.
\end{aligned}
\end{equation}
with,
\begin{equation}\label{pesos5_3}
\begin{aligned}
\tilde{\alpha}_{0,0}^6&=\frac{C^6_{0,0}}{(\epsilon+\tilde{I}_{0,0}^6)^t},\quad
\tilde{\alpha}_{0,1}^6=\frac{C^6_{0,1}}{(\epsilon+\tilde{I}_{0,1}^6)^t},\quad
\tilde{\alpha}_{1,1}^6=\frac{C^6_{1,1}}{(\epsilon+\tilde{I}_{1,1}^6)^t},\quad
\tilde{\alpha}_{1,2}^6=\frac{C^6_{1,2}}{(\epsilon+\tilde{I}_{1,2}^6)^t},\quad
\tilde{\alpha}_{2,2}^6=\frac{C^6_{2,2}}{(\epsilon+\tilde{I}_{2,2}^6)^t},\quad
\tilde{\alpha}_{2,3}^6=\frac{C^6_{2,3}}{(\epsilon+\tilde{I}_{2,3}^6)^t},\\
\tilde{\alpha}_{0,0}^6&=\frac{5/14}{(\epsilon+\tilde{I}_{0,0}^6)^t},\quad
\tilde{\alpha}_{0,1}^6=\frac{9/14} {(\epsilon+\tilde{I}_{0,1}^6)^t},\quad
\tilde{\alpha}_{1,1}^6=\frac{1/2}  {(\epsilon+\tilde{I}_{1,1}^6)^t},\quad
\tilde{\alpha}_{1,2}^6=\frac{1/2}  {(\epsilon+\tilde{I}_{1,2}^6)^t},\quad
\tilde{\alpha}_{0,0}^6=\frac{9/14} {(\epsilon+\tilde{I}_{2,2}^6)^t},\quad
\tilde{\alpha}_{0,1}^6=\frac{5/14} {(\epsilon+\tilde{I}_{2,3}^6)^t},\\
\tilde{\alpha}_{0,0}^7&=\frac{C^7_{0,0}}{(\epsilon+\tilde{I}_{0,0}^7)^t},\quad
\tilde{\alpha}_{0,1}^7=\frac{C^7_{0,1}}{(\epsilon+\tilde{I}_{0,1}^7)^t},\quad
\tilde{\alpha}_{1,1}^7=\frac{C^7_{1,1}}{(\epsilon+\tilde{I}_{1,1}^7)^t},\quad
\tilde{\alpha}_{1,2}^7=\frac{C^7_{1,2}}{(\epsilon+\tilde{I}_{1,2}^7)^t},\\
\tilde{\alpha}_{0,0}^7&=\frac{7/16}{(\epsilon+\tilde{I}_{0,0}^7)^t},\quad
\tilde{\alpha}_{0,1}^7=\frac{9/16}{(\epsilon+\tilde{I}_{0,1}^7)^t},\quad
\tilde{\alpha}_{1,1}^7=\frac{9/16}{(\epsilon+\tilde{I}_{1,1}^7)^t},\quad
\tilde{\alpha}_{1,2}^7=\frac{7/16}{(\epsilon+\tilde{I}_{1,2}^7)^t},\\
\tilde{\alpha}_{0,0}^8&=\frac{C^8_{0,0}}{(\epsilon+\tilde{I}_{0,0}^8)^t},\quad
\tilde{\alpha}_{0,1}^8=\frac{C^8_{0,1}}{(\epsilon+\tilde{I}_{0,1}^8)^t},\quad\\
\tilde{\alpha}_{0,0}^8&=\frac{1/2}{(\epsilon+\tilde{I}_{0,0}^8)^t},\quad
\tilde{\alpha}_{0,1}^8=\frac{1/2}{(\epsilon+\tilde{I}_{0,1}^8)^t}.
\end{aligned}
\end{equation}
where $\tilde{I}^l_{k,k_1}$, with $l=6,7,8$, $k=0,1,2$ and $k_1=k,k+1$ are the smoothness indicators that will be defined in Section \ref{generalsmoothness}. All in all, we get the new nonlinear weights,
\begin{equation}\label{v_pesos5}
\begin{split}
{\bf \tilde{C}^5}=(\tilde{C}_0^5,\tilde{C}_1^5,\tilde{C}_2^5,\tilde{C}_3^5,\tilde{C}_3^5)=&\tilde{\omega}_{0,0}^8\left(\tilde{\omega}_{0,0}^7\left(\tilde{\omega}_{0,0}^6{\bf C_0^6}+\tilde{\omega}_{0,1}^6{\bf C_1^6}\right)+\tilde{\omega}_{0,1}^7\left(\tilde{\omega}_{1,1}^6{\bf C_1^6}+\tilde{\omega}_{1,2}^6{\bf C_2^6}\right)\right)\\
&+\tilde{\omega}_{0,1}^8\left(\tilde{\omega}_{1,1}^7\left(\tilde{\omega}_{1,1}^6{\bf C_1^6}+\tilde{\omega}_{1,2}^6{\bf C_{2,2}^6}\right)+\tilde{\omega}_{1,2}^7\left(\tilde{\omega}_{2,2}^6{\bf C_2^6}+\tilde{\omega}_{2,3}^6{\bf C_3^6}\right)\right).
\end{split}
\end{equation}
Now we can apply the classical WENO method using these weights as optimal weights. In the next section, we will explain and prove the properties of the smoothness indicators used in the algorithm for any value of $r$.

\section{Smoothness indicators of high order}\label{smoothnessindicators}
In this section we present the smoothness indicators that we have chosen and that allow to obtain optimal order of accuracy close to the discontinuities.
We will use those introduced in \cite{smooth}, that work well for detecting kinks and jumps in the function if the data is discretized in the point values (\ref{pv}), we define them as:
\begin{equation}\label{si_nuestro2}
\tilde I_k^r=\sum_{l=2}^{r} h^{2l-1}\int_{x_{i-1}}^{x_{i}}\left(\frac{d^l}{dx^l}p^r_{k}(x)\right)^2 dx,\quad k=0,\hdots,r-1.
\end{equation}
For $r=2,3$, these weights satisfy several properties that allow to obtain the condition shown in Eq. \eqref{omega} (see \cite{smooth}). In this section we extend these results for any $r$.
In \cite{generalizacion} we introduced the proof of the following theorem that we reproduce here for completeness.
\begin{theorem}\label{ord_IS}
Let be $0\leq k\leq r-1$ and $p^r_k$ the interpolator polynomial of $f$ of degree $r\ge 3$ that uses the nodes of the stencil $S^r_k$, then at smooth zones, the smoothness indicator  obtained through (\ref{si_nuestro2}) satisfy,
$$\tilde{I}^r_k=\left(h^2 (p^r_k)''(x_{i-1/2})\right)^2\cdot(1+O(h^2)).$$
\end{theorem}
\vspace{-0.5cm}
\begin{proof}
Let be $0\leq k\leq r-1$ and $p^r_k$ the interpolator polynomial of $f$ of degree $r\ge 3$ that uses the nodes of the stencil $S^r_k$. It is an $r$ times continuously differentiable function in a neighborhood of the point $x_{i-1/2}=x_i-h/2$. Then, we can write $\left(p_k^{(l)}(x)\right)^2$ using the Taylor expansion of $p_k(x)$ as,
\begin{equation*}
\begin{aligned}
\left((p^r_k)^{(l)}(x)\right)^2&=\left((p^r_k)^{(l)}(x_{i-1/2})+(p^r_k)^{(l+1)}(x_{i-1/2})(x-x_{i-1/2})+O((x-x_{i-1/2})^2)\right)^2\\
&=\left((p^r_k)^{(l)}(x_{i-1/2})\right)^2+\left((p^r_k)^{(l+1)}(x_{i-1/2})(x-x_{i-1/2})\right)^2+2p_k^{(l)}(x_{i-1/2})(p^r_k)^{(l+1)}(x_{i-1/2})(x-x_{i-1/2})\\
&+2(p^r_k)^{(l)}(x_{i-1/2})(p^r_k)^{(l+2)}(x_{i-1/2})(x-x_{i-1/2})^2+O((x-x_{i-1/2})^3).
\end{aligned}
\end{equation*}
Integrating the previous expression between $x_{i-1}$ and $x_i$, replacing in (\ref{si_nuestro2}) and simplifying, we obtain the result, as for $l>2$ all the terms of the summation have a size smaller or equal than $O(h^6)$, $$\tilde I^r_k=\sum_{l=2}^{r} h^{2l-1}\int_{x_{i-1}}^{x_{i}}\left((p^r_k)^{(l)}(x)\right)^2 dx=\sum_{l=2}^{r} \left(h^{l}(p^r_k)^{(l)}(x_{i-1/2})\right)^2\cdot(1+O(h^2))=\left(h^2 (p^r_k)''(x_{i-1/2})\right)^2\cdot(1+O(h^2)).\qedhere$$
\end{proof}

The following proposition is proved in \cite{smooth}.
\begin{proposition}
Let be $0\leq k <r$,  $1\leq t$ and $\tilde{\omega}^r_k$ the nonlinear weights
 defined in Eq. \eqref{pesosr1finales} then:
\begin{equation*}
\begin{split}
&\tilde{\omega}^r_k=O(1), \quad \text{if $f$ is smooth in $S^r_k$,}\\
&\tilde{\omega}^r_k=O(h^{2mt}), \quad \text{if $f$ is not smooth in $S^r_k$.}\\
\end{split}
\end{equation*}
with $m=2$ if the discontinuity is in the function and $m=1$ if the discontinuity is in the first derivative.
\end{proposition}

We give the explicit form of $\tilde{I}^r_k$ with $k=0,\hdots,r-1$ for $r=3,4$. We represent the second order undivided differences by $\delta_{i}^2={ f_{i}}-2\,{ f_{i+1}}+{ f_{i+2}}$, $\delta_{i}^3=\delta_{i+1}^2-\delta_{i}^2$ and $\delta_{i}^4=\delta_{i+1}^3-\delta_{i}^3$. Taking into account the previous notation for the differences, we can write the smoothness indicators in (\ref{si_nuestro2}) for $r=3$ as,
\begin{equation}\label{IS4}
\begin{aligned}
\tilde I_{0}^3&=\frac{10}{3}(\delta^3_{i-3})^2+3 \delta^3_{i-3} \delta^2_{i-3}+(\delta^2_{i-3})^2,\\
\tilde I_{1}^3&=\frac{4}{3}(\delta^3_{i-2})^2+ \delta^3_{i-2} \delta^2_{i-2}+(\delta^2_{i-2})^2,\\
\tilde I_2^3&=\frac{4}{3}(\delta^3_{i-1})^2-\delta^3_{i-1} \delta^2_{i-1}+(\delta^2_{i-1})^2.
\end{aligned}
\end{equation}
For $r=4$ they can be written as:
\begin{equation}\label{IS5}
\begin{aligned}
\tilde I_{0}^4&={\frac {27}{2}}\,\delta^4_{i-4}\,\delta^3_{i-4}+\frac{11}{3}\,\delta^4_{i-4}\,\delta^2_{i-4}+5\,\delta^3_{i-4}\,\delta^2_{i-4}+{\frac {2107}{240}}\,(\delta^4_{i-4})^2+{\frac {22}{3}}\,(\delta^3_{i-4})^2+(\delta^2_{i-4})^2,\\
\tilde I_{1}^4&={\frac {547}{240}}(\delta^4_{i-3})^2+\frac{10}{3}\,(\delta^3_{i-3})^2+(\delta^2_{i-3})^2+{\frac {19}{6}}\,\delta^4_{i-3}\,\delta^3_{i-3}+\frac{2}{3}\,\delta^4_{i-3}\,\delta^2_{i-3}+3\,\delta^3_{i-3}\,\delta^2_{i-3},\\
\tilde I_{2}^4&={\frac {89}{80}}\,(\delta^4_{i-2})^2+\frac{4}{3}\,(\delta^3_{i-2})^2+(\delta^2_{i-2})^2-\frac{1}{6}\,\delta^4_{i-2}\,\delta^3_{i-2}-\frac{1}{3}\,\delta^4_{i-2}\,\delta^2_{i-2}+\delta^3_{i-2}\,\delta^2_{i-2},\\
\tilde I_{3}^4&={\frac {547}{240}}\,(\delta^4_{i-1})^2+\frac{4}{3}\,(\delta^3_{i-1})^2+(\delta^2_{i-1})^2-\frac{5}{2}\,\delta^4_{i-1}\,\delta^3_{i-1}+\frac{2}{3}\,\delta^4_{i-1}\,\delta^2_{i-1}-\delta^3_{i-1}\,\delta^2_{i-1}.
\end{aligned}
\end{equation}

Once defined the smoothness indicators for the level $r$, we will construct the rest of indicators for each level from $2r-2$ to $r+1$.
Theorem \ref{ord_IS} allows us to have a clearer vision about the behavior of the smoothness indicators. We can see that they get a value that is $O(h^4)$ at smooth zones. As they can be expressed in terms of differences, they get a value that is $O(1)$ when their stencil contains a discontinuity.

In \cite{generalizacion}, in order to reduce the computational cost of the calculation of smoothness indicators of high order through (\ref{si_nuestro2}), we proposed to use a function of the smoothness indicators used by the classical WENO algorithm. For $r=3$ we used as smoothness indicators of five points,
\begin{equation}\label{IS6}
\begin{aligned}
\tilde{I}_{0,0}^4&=\tilde{I}_{0}^3+\tilde{I}_{1}^3,\\
\tilde{I}_{0,1}^4&=\tilde{I}_{1}^3+\tilde I_{2}^3.
\end{aligned}
\end{equation}
These smoothness indicators provide optimal accuracy for $r=3$ in the sense that the pattern of accuracy obtained close to a discontinuity is $\cdots,O(h^6), O(h^5),O(h^4), O(1), O(h^4), O(h^5), O(h^6), \cdots$.
For $r=4$ we proposed to use as smoothness indicators of six and seven points $\tilde I_{0,0}^5, \tilde I_{0,1}^5, \tilde I_{1,2}^5, \tilde I_{0,0}^6, \tilde I_{0,1}^6$ defined as,
\begin{equation}\label{IS8}
\begin{aligned}
\tilde I_{0,0}^5&=\tilde I_{0}^4+\tilde I_{1}^4,\quad
\tilde I_{0,1}^5=\tilde I_{0,1}^5=\tilde I_{1}^4+\tilde I_{2}^4,\quad
\tilde I_{1,2}^5=\tilde I_{2}^4+\tilde I_{3}^4,\\
\tilde I_{0,0}^6&=\tilde I_{0}^4+I_{1}^4+\tilde I_{2}^4,\quad
\tilde I_{0,1}^6=\tilde I_{1}^4+\tilde I_{2}^4+\tilde I_{3}^4.
\end{aligned}
\end{equation}

In this case, these smoothness indicators do not provide optimal accuracy, providing the typical pattern of accuracy close to the discontinuity $\cdots, O(h^8), O(h^7), O(h^5), O(h^5), O(1), O(h^5), O(h^5), O(h^7), O(h^8), \cdots$ instead of the optimal one $\cdots, O(h^8), O(h^7), O(h^6), O(h^5), O(1), O(h^5), O(h^6), O(h^7), O(h^8), \cdots$. The reason is the following. Let's observe Figure \ref{disc_der1}, that represents the stencil of eight points considered for $r=4$. In this figure we have represented a discontinuity in the interval $(x_{i+1}, x_{i+2})$ (the case for the discontinuity in the interval $(x_{i-3}, x_{i-2})$ is symmetric). Let's now recall the expression of the interpolation proposed in \cite{generalizacion} and that we have reproduced in (\ref{pesos4_3}). The smoothness indicators $\tilde I_{0,0}^6$ and $\tilde I_{0,1}^6$ in (\ref{IS8}) use the data placed at the positions, $\{x_{i-4}, \cdots, x_{i+2}\}$ and $\{x_{i-3}, \cdots, x_{i+3}\}$ respectively. Looking at their expressions in (\ref{IS8}), it is clear that the weights $\tilde{\omega}_{0,0}^6$ and $\tilde{\omega}_{0,1}^6$ in (\ref{pesos4_2}) are both affected by the discontinuity, so both of them will have $O(1)$ accuracy. For a good approximation of the case represented in Figure \ref{disc_der1}, $\tilde{\omega}_{0,0}^6\approx 1$ and $\tilde{\omega}_{0,1}^6\approx 0$ in order to keep the final weights close to the optimal ones. The problem would be solved if $\tilde{\omega}_{1,1}^5\approx 0$ and $\tilde{\omega}_{1,2}^5\approx 0$ but this is not possible, as the algorithm balances the weights in pairs of two, thus, if one is close to zero the other one must be close to one. The problem is solved in this case if we select the smoothness indicators in pairs as,
\begin{equation}\label{IS8_2}
\begin{split}
\tilde I_{0,0}^6&=\tilde I_{0}^4,\quad \tilde I_{0,1}^6=\tilde I_{3}^4,\\
\tilde I_{0,0}^5&=\tilde I_{0}^4,\quad \tilde I_{0,1}^5=\tilde I_{2}^4,\\
\tilde I_{1,1}^5&=\tilde I_{1}^4,\quad \tilde I_{1,2}^5=\tilde I_{3}^4.
\end{split}
\end{equation}
Now it is clear that $\tilde{\omega}_{0,0}^6$ {\it watches} the discontinuity and it gets a value $\tilde{\omega}_{0,0}^6\approx 1$ while $\tilde{\omega}_{0,1}^6\approx 0$, cleaning the garbage introduced if we use (\ref{IS8}) and similarity for $\tilde{\omega}_{k,k_1}^5$ with $k=0,1, k_1=k+1$. With this strategy we only compute the smoothness indicators of degree $r$, that are the ones used by the classical WENO algorithm, improving the computational efficiency.

\begin{figure}[!ht]
\begin{center}
\resizebox{12cm}{!} {
\begin{picture}(500,30)(0,0)
\put(0,-10){\line(1,0){560}}

\put(0,-10){\line(0,1){5}}
\put(80,-10){\line(0,1){5}}
\put(160,-10){\line(0,1){5}}
\put(240,-10){\line(0,1){5}}
\put(320,-10){\line(0,1){5}}
\put(400,-10){\line(0,1){5}}
\put(480,-10){\line(0,1){5}}
\put(560,-10){\line(0,1){5}}
\put(440,-12){$\times$}
\put(-5,5){$f_{i-4}$}
\put(75,5){$f_{i-3}$}
\put(155,5){$f_{i-2}$}
\put(235,5){$f_{i-1}$}
\put(315,5){$f_{i}$}
\put(395,5){$f_{i+1}$}
\put(475,5){$f_{i+2}$}
\put(555,5){$f_{i+3}$}

\put(-5,-23){$x_{i-4}$}
\put(75,-23){$x_{i-3}$}
\put(155,-23){$x_{i-2}$}
\put(235,-23){$x_{i-1}$}
\put(315,-23){$x_{i}$}
\put(395,-23){$x_{i+1}$}
\put(475,-23){$x_{i+2}$}
\put(555,-23){$x_{i+3}$}

\end{picture}
}
\end{center}
\caption{Representation of the values considered in the stencil and a discontinuity placed in the interval $[x_{i+1}, x_{i+2}]$.}\label{disc_der1}
\end{figure}
In the following subsection we give a general formula to compute the smoothness indicator for any $r$.
\subsection{General smoothness indicators}\label{generalsmoothness}

In this subsection we provide a general formula to obtain any smoothness indicator $I_{k,k1}^{l}$ in terms of the classical WENO smoothness indicators $\tilde I^r_k$ for $r$. We follow the same construction presented in all the section. We define the different weights using the following formulas,
\begin{definition}\label{def1}
Let be $l=r+1,\dots, 2r-2$, and $\tilde I^r_k$, with $k=0,\hdots,r-1$ the smoothness indicators showed in Eq. \eqref{si_nuestro2}, then we define
the smoothness indicators at level $l$ as:
\begin{equation}\label{equationindicadores}
\begin{split}
&\tilde I_{k,k}^{l}=\tilde I_k^r, \quad k=0,\dots,(2r-2)-l,\\
&\tilde I_{k,k+1}^{l}=\tilde I_{l-(r-1)+k}^r, \quad k=0,\dots,(2r-2)-l.
\end{split}
\end{equation}
\end{definition}

The idea is to detect in what part of the tree diagram is placed the discontinuity and to force that this branch is automatically discarded through the values of the smoothness indicators and, consequently, the weights. Then continue with the following sub-branch and again force the automatic discarding of this branch if the stencil contains a discontinuity. We repeat this process in each pass until we obtain only one branch free of discontinuities. 

We have already described the smoothness indicators for $r=3$ and $r=4$.  For $r=5$, using Eq. \eqref{equationindicadores} we get:

\begin{equation}\label{IS10}
\begin{split}
\tilde I_{0,0}^8&=\tilde I_{0}^5,\quad \tilde I_{0,1}^8=\tilde I_{4}^5, \\
\tilde I_{0,0}^7&=\tilde I_{0}^5,\quad \tilde I_{0,1}^7=\tilde I_{3}^5,\quad \tilde I_{1,1}^7=\tilde I_{1}^5, \quad  \tilde I_{1,2}^7=\tilde I_{4}^5,\\
\tilde I_{0,0}^6&=\tilde I_{0}^5,\quad \tilde I_{0,1}^6=\tilde I_{2}^5,\quad \tilde I_{1,1}^6=\tilde I_{1}^5, \quad \tilde I_{1,2}^6=\tilde I_{3}^5.\\
\end{split}
\end{equation}

\begin{figure}[!ht]
\begin{center}
\resizebox{11cm}{!} {
\begin{picture}(500,30)(0,0)
\put(-80,-10){\line(1,0){720}}

\put(-80,-10){\line(0,1){5}}
\put(0,-10){\line(0,1){5}}
\put(80,-10){\line(0,1){5}}
\put(160,-10){\line(0,1){5}}
\put(240,-10){\line(0,1){5}}
\put(320,-10){\line(0,1){5}}
\put(400,-10){\line(0,1){5}}
\put(480,-10){\line(0,1){5}}
\put(560,-10){\line(0,1){5}}
\put(640,-10){\line(0,1){5}}
\put(520,-12){$\times$}
\put(440,-12){$\times$}
\put(-85,5){$f_{i-5}$}
\put(-5,5){$f_{i-4}$}
\put(75,5){$f_{i-3}$}
\put(155,5){$f_{i-2}$}
\put(235,5){$f_{i-1}$}
\put(315,5){$f_{i}$}
\put(395,5){$f_{i+1}$}
\put(475,5){$f_{i+2}$}
\put(555,5){$f_{i+3}$}
\put(635,5){$f_{i+4}$}
\put(-85,-23){$x_{i-4}$}
\put(-5,-23){$x_{i-4}$}
\put(75,-23){$x_{i-3}$}
\put(155,-23){$x_{i-2}$}
\put(235,-23){$x_{i-1}$}
\put(315,-23){$x_{i}$}
\put(395,-23){$x_{i+1}$}
\put(475,-23){$x_{i+2}$}
\put(555,-23){$x_{i+3}$}
\put(635,-23){$x_{i+4}$}
\end{picture}
}
\end{center}
\caption{Representation of the values considered in the stencil and a discontinuity placed in the interval $[x_{i+1}, x_{i+2}]$.}\label{disc_der2}
\end{figure}


Finally, we prove the following important results that we will use in the next section in order to analyze the accuracy of the new method.

\begin{lemma}\label{lemaaux1}
Let be $0\leq k,k_1 \leq r-2$ and $\tilde{I}^r_{n}$, $n=k,k_1$  be smoothness indicators of $f$ on the stencil $S^r_n=\{x_{i+n-r},\cdots, x_{i+n}\}$. If $f \in C^r([x_{i+n-r},x_{i+n}])$, then
\begin{equation*}
\tilde{I}_k^r-\tilde{I}_{k_1}^r=O(h^{r+3}),
\end{equation*}
\end{lemma}
\begin{proof}
Let $p_{k}^r,p_{k_1}^r$ be the two interpolating polynomials of $f$ of degree $r\geq 3$ at nodes in the stencil $S_{n}^r$ $n=k,k_1$ at a distance $O(h)$, then
if $l\geq 2$:

We follow the sketch of the proof presented in \cite{AMB}, thus from $(f^{(l)})- (p_{n}^r)^{(l)}=O(h^{r+1-l})$, we get
\begin{equation}
\begin{split}
(f^{(l)})^2- ((p^r_{n})^{(l)})^2&=-\left((f^{(l)}- (p^r_{n})^{(l)})^2+2((p^r_{n})^{(l)}-f^{(l)})f^{(l)}\right)\\
&=(O(h^{r+1-l}))^2 +(O(h^{r+1-l}))O(1)\\
&=O(h^{r+1-l}),\\
\end{split}
\end{equation}
if we substitute in \ref{si_nuestro2}, we obtain
\begin{equation*}
\begin{split}
\left|\tilde I_{n}^r-\sum_{l=2}^{r} h^{2l-1}\int_{x_{i-1}}^{x_{i}}\left(f^{(l)}(x)\right)^2 dx\right|
&\leq\left|\sum_{l=2}^{r} h^{2l-1}\int_{x_{i-1}}^{x_{i}}\left((p^r_{n})^{(l)}(x)\right)^2-\left(f^{(l)}(x)\right)^2 dx\right|\\
&= \sum_{l=2}^{r} h^{2l}O(h^{r+1-l})\\
&= O(h^{r+3}).
\end{split}
\end{equation*}
Finally,
\begin{equation}
\begin{split}
\left|\tilde I_{k}^r-\tilde I_{k_1}^r\right|&\leq \left|\tilde I_{k}^r-\sum_{l=2}^{r} h^{2l-1}\int_{x_{i-1}}^{x_{i}}\left(f^{(l)}(x)\right)^2 dx\right|+\left|\tilde I_{k+1}^r-\sum_{l=2}^{r} h^{2l-1}\int_{x_{i-1}}^{x_{i}}\left(f^{(l)}(x)\right)^2 dx\right|\\
&= O(h^{r+3}).
\end{split}
\end{equation}
\end{proof}

\begin{proposition}\label{prop1}
Let be $r+1\leq l\leq 2r-2$, $0\leq k\leq (2r-2)-l$ and $\tilde I^l_{k,k}$ and $\tilde I^l_{k,k+1}$ be smoothness indicators defined in Def. \ref{def1}, that can be expressed, following Theorem \ref{ord_IS}, through Taylor expansion as $\tilde I^l_{k,k}=\left(h^2 p''(x_{i-1/2})\right)^2(1+O(h^2))$ if the stencil is smooth, $\tilde I^l_{k,k}=O(1)$ if the stencil is affected by a discontinuity in the function and $\tilde I^l_{k,k}=O(h)$ if the stencil is affected by a discontinuity in the first derivative. In this case, any weights expressed as
\begin{equation}\label{w_teo1}
\begin{aligned}
\tilde{\omega}^l_{k,k}=\frac{\tilde{\alpha}_{k,k}^l}{\tilde{\alpha}_{k,k}^l+\tilde{\alpha}_{k,k+1}^l},\quad
\tilde{\omega}^l_{k,k+1}=\frac{\tilde{\alpha}_{k,k+1}^l}{\tilde{\alpha}_{k,k}^l+\tilde{\alpha}_{k,k+1}^l},
\end{aligned}
\end{equation}
with,
\begin{equation}\label{alpha_teo1}
\begin{aligned}
\tilde{\alpha}_{k,k}^l=\frac{C_{k,k}^l}{(\epsilon+\tilde{I}_{k,k}^l)^t},\quad
\tilde{\alpha}_{k,k+1}^l=\frac{C_{k,k+1}^l}{(\epsilon+\tilde{I}_{k,k+1}^l)^t}.
\end{aligned}
\end{equation}
with $C^l_{k,k}+C^l_{k,k+1}=1$, will fall under one of the following cases:
\begin{enumerate}
\item If neither $\tilde{I}^l_{k,k}$ nor $\tilde I^l_{k,k+1}$ are affected by a discontinuity, then $\tilde{\omega}_{k,k}^l=C^l_{k,k}+O(h^{r-1})$ and $\tilde{\omega}_{k,k+1}^l=C^l_{k,k+1}+O(h^{r-1})$.
\item If $\tilde{I}^l_{k,k+1}$ is affected by a singularity, then $\tilde{\omega}_{k,k}^l=1+O(h^{2mt})$ and $\tilde{\omega}_{k,k+1}^l=O(h^{2mt})$.
\item If $\tilde{I}^l_{k,k}$ is affected by a singularity then $\tilde{\omega}_{k,k+1}^l=1+O(h^{2mt})$ and $\tilde{\omega}_{k,k}^l=O(h^{2mt})$.
\item If $\tilde{I}^l_{k,k}$ and $\tilde{I}^l_{k,k+1}$ are affected by a singularity then $\tilde{\omega}_{k,k}^l=O(1)$ and $\tilde{\omega}_{k,k+1}^l=O(1)$.
\end{enumerate}
With $m=2$ if the discontinuity is in the function and $m=1$ if the discontinuity is in the first derivative.
\end{proposition}
\begin{proof}
We follow the ideas presented in \cite{AMB}. Let be $r+1\leq l \leq (2r-2)$ and $0\leq k\leq  (2r-2)-l$, by Def. \ref{def1}:
\begin{equation}
\begin{split}
&\tilde I_{k,k}^{l}=\tilde I_k^r, \quad \tilde I_{k,k+1}^{l}=\tilde I_{l-(r-1)+k}^r=\tilde I_{k_1}^r,
\end{split}
\end{equation}
being $k_1=l-(r-1)+k$, doing algebraic manipulations, we have
$$\frac{\frac{1}{(\epsilon+\tilde{I}_k^r)^t}+\frac{1}{(\epsilon+\tilde{I}_{k_1}^r)^t}}{\frac{1}{(\epsilon+\tilde{I}_{k_1}^r)^t}}
=\frac{\tilde{I}_{k_1}^r-\tilde{I}_{k}^r}{\epsilon+\tilde{I}_{k}^r}\sum_{j=0}^{t-1}\left(\frac{\epsilon+\tilde{I}^r_{k_1}}{\epsilon+\tilde{I}^r_{k}}\right)^j,$$
then, using Lemma \ref{lemaaux1}, the previous equality transforms into,
\begin{equation}\label{eq1}
\frac{1}{(\epsilon+\tilde{I}_k^r)^t}=\frac{1+O(h^{r-1})}{(\epsilon+\tilde{I}_{k_1}^r)^t}.
\end{equation}
If there is no singularity affecting $\tilde I^r_{k,k}=\tilde{I}_k^r$ or $\tilde I^r_{k,k+1}=\tilde{I}_{k_1}^r$ then, using Eq. \eqref{eq1}, we have
\begin{equation}\label{caso1}
\begin{split}
\tilde{\omega}^l_{k,k}&=\frac{\frac{C_{k,k}^l}{(\epsilon+\tilde I_{k,k}^l)^t}}{\frac{C^l_{k,k}}{(\epsilon+\tilde I_{k,k}^l)^t}+\frac{C^l_{k,k+1}}{(\epsilon+\tilde I_{k,k+1}^l)^t}}=\frac{\frac{C_{k,k}^l}{(\epsilon+\tilde I_{k}^r)^t}}{\frac{C_{k,k}^l}{(\epsilon+\tilde I_{k}^r)^t}+\frac{\tilde C_{k,k+1}^l}{(\epsilon+\tilde I_{k_1}^r)^t}}=\frac{\frac{C_{k,k}^l(1+O(h^{r-1}))}{(\epsilon+\tilde{I}_{k_1}^r)^t}}{\frac{C_{k,k}^l(1+O(h^{r-1}))}{(\epsilon+\tilde{I}_{k_1}^r)^t}+\frac{C^l_{k,k+1}}{(\epsilon+\tilde I_{k_1}^r)^t}}=\frac{C_{k,k}^l(1+O(h^{r-1}))}{C_{k,k}^l(1+O(h^{r-1}))+C^l_{k,k+1}}\\
 &=C_{k,k}^l+O(h^{r-1}).\\
\text{Analogously,}&\\
\tilde{\omega}^l_{k,k+1}&=C_{k,k+1}^l+O(h^{r-1}).
\end{split}
\end{equation}
If, following Theorem \ref{ord_IS},  $\tilde I^l_{k,k}=K(1+O(h^2))$ if the stencil is smooth and $\tilde I^l_{k,k}=O(1)$ if the stencil is affected by a discontinuity and we consider that $\epsilon$ is small enough, then,
\begin{enumerate}
\item If the stencil contains a singularity that only affects the smoothness indicator $\tilde I^l_{k,k+1}$, then $\tilde I^l_{k,k}=\left(h^2 p''(x_{i-1/2})\right)^2(1+O(h^2))=O(h^{2m})$ with $m=2$ if the discontinuity is in the function and $m=1$ if the discontinuity is in the first derivative,  and $\tilde I^l_{k,k+1}=O(1)$. Then,
\begin{equation}\label{caso2}
\begin{split}
\tilde{\omega}^l_{k,k}&=\frac{C_{k,k}^l(\epsilon+\tilde I_{k,k+1}^l)^t}{C_{k,k}^l(\epsilon+\tilde I_{k,k+1}^l)^t+C_{k,k+1}^l(\epsilon+\tilde I_{k,k}^l)^t}=\frac{C_{k,k}^l(\epsilon+\tilde I_{k,k+1}^l)^t}{C_{k,k}^l(\epsilon+\tilde I_{k,k+1}^l)^t+O(h^{2mt})}=\frac{C_{k,k}^l(\epsilon+\tilde I_{k,k+1}^l)^t}{C_{k,k}^l(\epsilon+\tilde I_{k,k+1}^l)^t}\frac{1}{1+\frac{O(h^{2mt})}{C_{k,k}^l(\epsilon+\tilde I_{k,k+1}^l)^t}}\\
&=1+O(h^{2mt}),\\
\tilde{\omega}^l_{k,k+1}&=\frac{C_{k,k+1}^l(\epsilon+\tilde I_{k,k}^l)^t}{C_{k,k}^l(\epsilon+\tilde I_{k,k+1}^l)^t+C_{k,k+1}^l(\epsilon+\tilde I_{k,k}^l)^t}=\frac{O(h^{2mt})}{O(1)+O(h^{2mt})}=O(h^{2mt}).
\end{split}
\end{equation}
\item The case when the stencil contains a singularity that only affects the smoothness indicator $I_{k,k}$ can be obtained by symmetry.
\item If $\tilde I^l_{k,k}$ and $\tilde I^l_{k,k+1}$ are affected by a singularity then,
\begin{equation}\label{caso4}
\begin{aligned}
\tilde{\omega}^l_{k,k}&=\frac{C_{k,k}^l(\epsilon+\tilde I_{k,k+1}^l)^t}{C_{k,k}^l(\epsilon+\tilde I_{k,k+1}^l)^t+C_{k,k+1}^l(\epsilon+\tilde I_{k,k}^l)^t}=\frac{C_{k,k}^l(\epsilon+\tilde I_{k,k+1}^l)^t}{C_{k,k}^l(\epsilon+\tilde I_{k,k+1}^l)^t+O(h^{2mt})}=\frac{C_{k,k}^l(\epsilon+\tilde I_{k,k+1}^n)^t}{C_{k,k}^n(\epsilon+\tilde I_{k,k+1}^l)^t}\frac{1}{1+\frac{O(h^{2mt})}{O(h^{2mt})}}=O(1),\\
\tilde{\omega}^l_{k,k+1}&=\frac{C_{k,k+1}^l(\epsilon+\tilde I_{k,k}^l)^t}{C_{k,k}^l(\epsilon+\tilde I_{k,k+1}^l)^t+C_{k,k+1}^n(\epsilon+\tilde I_{k,k}^l)^t}=O(1).
\end{aligned}
\end{equation}

\end{enumerate}
\end{proof}


\section{Analysis of the accuracy}\label{analisisaccuracy}

Now we can try to analyze the accuracy of each individual dyadic WENO algorithm that we perform at every step of the new construction. In general, to prove the accuracy of this algorithm we will need three steps:
\begin{enumerate}
\item Obtaining the value that the nonlinear optimal weights get for each position of the discontinuity.
\item Checking the error obtained between the weights of the classical WENO (that is performed at the last step) obtained using the nonlinear optimal weights and the weights that provide optimal accuracy.
\item Checking the accuracy obtained by classical WENO interpolation.
\end{enumerate}

\subsection{Analysis of the accuracy for any $r$}\label{section2}

We present in this section a general study of the accuracy of the new algorithm using the formulas described in Section \ref{newweno}. Therefore,
the weights are given by:
\begin{equation*}
\begin{split}
(\tilde{C}_0^r,\tilde{C}_1^r,\hdots,\tilde{C}_{r-2}^r,\tilde{C}_{r-1}^r)=\sum_{j_0=0}^1 \tilde{\omega}^{2r-2}_{0,j_0}\left(\sum_{j_1=j_0}^{j_0+1}\tilde{\omega}_{j_0,j_1}^{2r-3}\left(\sum_{j_2=j_1}^{j_1+1}\tilde{\omega}_{j_1,j_2}^{2r-4}
\left(\dots\left(\sum_{j_{r-2}=j_{r-3}}^{j_{r-3}+1} \tilde{\omega}^{r+1}_{j_{r-3},j_{r-2}}{\bf  C_{j_{r-2}}^{r+1}} \right)\dots\right)\right)\right)
\end{split}
\end{equation*}
where the values $\tilde{\omega}_{k,k_1}^l$ with $l=r+1,\dots,2r-2$; $0\leq k \leq (2r-2)-l$; $k_1=k+1$ and ${\bf C_{k}^{r+1}}$ being $0\leq k\leq r-2$ are defined in Eqs. \eqref{pesosr} and \eqref{nl_op_w_r}.  The coordinates of the vector $(\tilde \omega^r_0,\hdots,\tilde \omega^r_{r-1})$ have been defined in Eq. \eqref{pesosr1finales}, i.e.
\begin{equation}
\begin{split}
&\tilde \omega^r_{k}=\frac{\tilde \alpha_{k}^r}{\sum_{s=0}^{r-1}\tilde \alpha_s^r},\quad k=0,\dots,r-1,\\
&\tilde \alpha_{k}^r=\frac{\tilde C_{k}^r}{(\epsilon+\tilde I^r_{k})^t}, \quad k=0,\dots,r-1.\\
\end{split}
\end{equation}
We analyze the different possibilities,
\begin{itemize}
\item Firstly, we suppose that any discontinuity does not cross the stencil $\{x_{i-r},\dots,x_{i+r-1}\}$ then by Prop. \ref{prop1}, we get for all $l=r+1,\dots,2r-2$, $0\leq k \leq (2r-2)-l$ and $k_1=k+1$ that
$$\tilde{\omega}_{k,k_1}^l = \tilde C^l_{k,k_1}+O(h^{r-1}).$$
Then, by construction of the optimal weights, Eq. \eqref{equationimpo},
\begin{equation*}
\begin{split}
(\tilde \omega_0^r,\hdots,\tilde \omega_{r-1}^r)&=(\tilde{C}_0^r,\hdots,\tilde{C}_{r-1}^r)+O(h^{r-1})\\
&=\sum_{j_0=0}^1 \tilde{\omega}^{2r-2}_{0,j_0}\left(\sum_{j_1=j_0}^{j_0+1}\tilde{\omega}_{j_0,j_1}^{2r-3}\left(\sum_{j_2=j_1}^{j_1+1}\tilde{\omega}_{j_1,j_2}^{2r-4}
\left(\dots\left(\sum_{j_{r-2}=j_{r-3}}^{j_{r-3}+1} \tilde{\omega}^{r+1}_{j_{r-3},j_{r-2}}{\bf C_{j_{r-2}}^{r+1}} \right)\dots\right)\right)\right)+O(h^{r-1})\\
&=\sum_{j_0=0}^1 C^{2r-2}_{0,j_0}\left(\sum_{j_1=j_0}^{j_0+1}C_{j_0,j_1}^{2r-3}\left(\sum_{j_2=j_1}^{j_1+1}C_{j_1,j_2}^{2r-4}
\left(\dots\left(\sum_{j_{r-2}=j_{r-3}}^{j_{r-3}+1} C^{r+1}_{j_{r-3},j_{r-2}}{\bf C_{j_{r-2}}^{r+1}} \right)\dots\right)\right)\right)+O(h^{r-1})\\
&=(\bar{C}_0^r,\hdots,\bar{C}_{r-1}^r)+O(h^{r-1})\\
\end{split}
\end{equation*}

\item If there exists a discontinuity at $[x_{i-1},x_{i}]$ by Prop. \ref{prop1}, for all  $l=r+1,\dots,2r-2$, $0\leq k \leq (2r-2)-l$ and $k_1=k+1$ we get:
    $$\tilde{\omega}_{k,k_1}^l=O(1).$$

\item By symmetry, we only analyze when there exists an isolated discontinuity at an interval $[x_{i-1+l_0},x_{i+l_0}]$, $l_0=1,\hdots, r-1$ (analogously, we obtain the equivalent symmetric results for $[x_{i-r+l_0},x_{i-r+l_0+1}]$, $l_0=0,\hdots,r-2$).  In order to study these cases, we prove the following results.
\end{itemize}

\begin{lemma}\label{lemmaauxi}
Let be $\tilde{\omega}^r_k$, $k=0,\hdots,r-1$, be the nonlinear weights
 defined in Eq. \eqref{pesosr1finales}. If there exists $0\leq l_0\leq r-1$ such that
 $I_{l_0}^r$ is affected by a discontinuity and
 $$(\tilde{C}^r_0,\hdots,\tilde{C}^r_{r-1})=(C^r_0+O(h^s),\hdots,C^r_{l_0-1}+O(h^s),O(h^{2mt}),\hdots,O(h^{2mt}))$$
 with $\sum_{k=0}^{l_0-1}C^r_k=1$, $1\leq s\leq 2mt$, $m=2$ if the discontinuity is in the function and $m=1$ if the discontinuity is in the first derivative, then:
\begin{equation*}
\begin{split}
&\tilde{\omega}^r_{k}=C^r_k+O(h^s), \quad 0\leq k<l_0,\\
&\tilde{\omega}^r_k=O(h^{2mt}), \quad l_0\leq k \leq r-1.\\
\end{split}
\end{equation*}
\end{lemma}
\begin{proof}
This is a direct consequence of the fact that $\sum_{k=0}^{l_0-1}C^r_k=1$.
\end{proof}

\begin{lemma}\label{lema1}
Let be  $0<l_0\leq r-1$. If there exists a discontinuity at $[x_{i+l_0-1},x_{i+l_0}]$, then for all $l_0+(r-1)\leq l \leq 2r-2 $ the nonlinear weights
defined in Eq. \eqref{pesosr} satisfy:
\begin{equation}
\begin{split}
&\tilde{\omega}^l_{0,0}=1+O(h^{2mt}),\quad  \tilde{\omega}^l_{0,1}=O(h^{2mt})
\end{split}
\end{equation}
With $m=2$ if the discontinuity is in the function and $m=1$ if the discontinuity is in the first derivative.
\end{lemma}

\begin{proof}
It is clear that as $l_0>0$, then the smoothness indicator $\tilde I^r_0$ is not affected by the discontinuity. However, if $l_0+(r-1)\leq l  \leq 2r-2$, then:
$$ l-(r-1)-r < 0<l_0= l_0+(r-1)-(r-1)\leq l-(r-1). $$
Thus, the discontinuity crosses the stencil used to calculate $\tilde I_{l-(r-1)}^r$ then, by the definition of the smoothness indicators in \eqref{equationindicadores}, we get:
\begin{equation*}
\tilde I_{0,0}^{l}=\tilde I_0^r, \quad \tilde I_{0,1}^{l}=\tilde I_{l-(r-1)}^r,
\end{equation*}
and by Prop. \ref{prop1} the result is obtained, i.e.
 \begin{equation*}
\begin{split}
&\tilde{\omega}^l_{0,0}=1+O(h^{2mt}),\quad  \tilde{\omega}^l_{0,1}=O(h^{2mt}).
\end{split}
\end{equation*}
\end{proof}

\begin{lemma}\label{lema2}
Let be  $0<l_0\leq r-1$. If there exists a discontinuity at $[x_{i+l_0-1},x_{i+l_0}]$, then for all $r+1\leq l \leq l_0+(r-2)$ the nonlinear weights
defined in Eq. \eqref{pesosr} satisfy:
\begin{equation}
\begin{split}
&\tilde{\omega}^l_{k,k}=C^l_{k,k}+O(h^{r-1}),\quad  \tilde{\omega}^l_{k,k+1}=C^l_{k,k+1}+O(h^{r-1}),\quad 0\leq k \leq l_0+(r-2)-l,
\end{split}
\end{equation}
being $C^l_{k,k_1}$ with $k_1=k,\,\,k+1$ defined in Eq. \eqref{pesostodos}.
\end{lemma}

\begin{proof}
In order to prove this lemma, we only have to analyze if the discontinuity crosses the stencils used to calculate
$\tilde I^{r}_k$ and $\tilde I^r_{l-(r-1)+k}$. In the first case, the stencil is $\{x_{i-r+k},\dots,x_{i+k}\}$, as
$$k \leq l_0+(r-2)-l \leq l_0+(r-2)-(r+1) = l_0-3,$$
then the discontinuity does not cross it.
In second case the stencil is $\{x_{i+l-(r-1)+k-r},\dots,x_{i+l-(r-1)+k}\}$, from
$$l-(r-1)+k\leq l-(r-1)+l_0+(r-2)-l=l_0-1. $$
Using Prop. \ref{prop1} we get the result.
\end{proof}

\begin{lemma}\label{lema_t}
Let be  $1<l_0\leq r-1$. If there exists a discontinuity at $[x_{i+l_0-1},x_{i+l_0}]$, then
\begin{equation*}
\begin{split}
(\tilde{C}_0^r,\tilde{C}_1^r,\hdots,\tilde{C}_{r-1}^r)=\sum_{j_0=0}^1 \tilde{\omega}^{r-2+l_0}_{0,j_{0}}\left(\sum_{j_{1}=j_{0}}^{j_0+1}\tilde{\omega}_{j_0,j_1}^{r-3+l_0}\left(\dots\left(\sum_{j_{l_0-2}=j_{l_0-3}}^{j_{l_0-3}+1} \tilde{\omega}^{r+1}_{j_{l_0-3},j_{l_0-2}}{\bf  C_{j_{l_0-2}}^{r+1}} \right)\dots\right)\right)+O(h^{2mt})
\end{split}
\end{equation*}
with ${\bf  C_{k}^{r+1}}$, $k=0,\dots,r-1$ defined in Eq. \eqref{nl_op_w_r} and
\begin{equation}
\begin{split}
&\tilde{\omega}^l_{k,k_1}=\frac{\tilde{\alpha}_{k,k_1}^l}{\tilde{\alpha}_{k,k}^l+\tilde{\alpha}_{k,k+1}^l},\quad \quad\tilde{\alpha}_{k,k_1}^l=\frac{C_{k,k_1}^l}{(\epsilon+\tilde I^l_{k,k_1})^t}, \quad k_1=k,\,\,k+1.\\
\end{split}
\end{equation}
where $\tilde I^l_{k,k_1}$ are the smoothness indicators defined in Section \ref{generalsmoothness}, $C_{k,k_1}^l$ defined in Eq. \eqref{pesostodos}, $m=2$ if the discontinuity is in the function and $m=1$ if the discontinuity is in the first derivative.

Also,
\begin{equation}\label{final}
(\tilde\omega_0^r,\tilde\omega_1^r,\hdots,\tilde\omega_{r-1}^r)=(\hat{C}_0^r+O(h^{r-1}),\hat{C}_1^r+O(h^{r-1}),\hdots,\hat{C}_{l_0-1}^r+O(h^{r-1}),O(h^{2mt}),\hdots,O(h^{2mt}))
\end{equation}
being
$$p_0^{l_0+r-1}(x_{i-1/2})=\sum_{k=0}^{l_0-1}\hat{C}_k^r p^r_k(x_{i-1/2})$$
\end{lemma}

\begin{proof}
We suppose that $1<l_0\leq r-1$, we denote as:
\begin{equation}
\begin{split}
p^{l_0+r-1}_0(x_{i-1/2})&=\sum_{j_0=0}^1C^{l_0+r-2}_{0,j_0}p_{j_0}^{l_0+r-2}(x_{i-1/2}) \\
&=\sum_{j_0=0}^1C^{l_0+r-2}_{0,j_0}\left(\sum_{j_1=j_0}^{j_0+1} C^{l_0+r-3}_{j_0,j_1} p_{j_1}^{l_0+r-3}(x_{i-1/2})\right)\\
&=\sum_{j_0=0}^1 C^{l_0+r-2}_{0,j_0}\left(\sum_{j_1=j_0}^{j_0+1}C_{j_0,j_1}^{l_0+r-3}
\left(\dots\left(\sum_{j_{l_0-2}=j_{l_0-3}}^{j_{l_0-3}+1} C^{r+1}_{j_{l_0-3},j_{l_0-2}}\left(\sum_{j_{l_0-1}=j_{l_0-2}}^{j_{l_0-2}+1} C^{r}_{j_{l_0-2},j_{l_0-1}}p^r_{j_{l_0-1}}(x_{i-1/2})\right)\right)\dots\right)\right)\\
&=\sum_{k=0}^{l_0-1}\hat{C}_k^r p_k^r(x_{i-1/2}).
\end{split}
\end{equation}
Then, by Eq. \eqref{eqsuperimp} and Lemma \ref{lema1} we have that:
\begin{equation*}
\begin{split}
(\tilde{C}_0^r,\tilde{C}_1^r,\hdots,\tilde{C}_{r-1}^r)=\sum_{j_0=0}^1 \tilde{\omega}^{r-2+l_0}_{0,j_{0}}\left(\sum_{j_{1}=j_{0}}^{j_0+1}\tilde{\omega}_{j_0,j_1}^{r-3+l_0}\left(\dots\left(\sum_{j_{l_0-2}=j_{l_0-3}}^{j_{l_0-3}+1} \tilde{\omega}^{r+1}_{j_{l_0-3},j_{l_0-2}}{\bf  C_{j_{l_0-2}}^{r+1}} \right)\dots\right)\right)+O(h^{2mt})
\end{split}
\end{equation*}
And by Lemmas  \ref{lemmaauxi} and \ref{lema2}, we obtain:
\begin{equation*}
(\tilde \omega_0^r,\tilde \omega_1^r,\hdots,\tilde \omega_{r-1}^r)=(\hat{C}_0^r+O(h^{r-1}),\hat{C}_1^r+O(h^{r-1}),\hdots,\hat{C}_{l_0-1}^r+O(h^{r-1}),O(h^{2mt}),\hdots,O(h^{2mt})).
\end{equation*}
\end{proof}

\begin{theorem}\label{teo1}
Let be $1< l_0 \leq r-1$ and $\tilde \omega_k^r$ defined in Eq. \eqref{pesosr1finales},  if $f$ is smooth in $[x_{i-r},x_{i+r-1}]\setminus \Omega$ and $f$ has a discontinuity at $\Omega$ then
\begin{equation}
\sum_{k=0}^{r-1}\tilde \omega^r_k p^r_k(x_{i-\frac12})-{f}(x_{i-\frac12})=\left\{
                                                  \begin{array}{ll}
                                                    O(h^{2r}), & \hbox{if $\,\,\Omega=\emptyset$;} \\
O(h^{r+l_0}), & \hbox{if   $\,\,\Omega=[x_{i+l_0-1},x_{i+l_0}]$; }\\
                                                  \end{array}
                                                \right.
\end{equation}
\end{theorem}

\begin{proof}
Let be $1 < l_0 \leq r-1$, then
\begin{equation*}
\begin{split}
\sum_{k=0}^{r-1}\tilde \omega^r_k p^r_k(x_{i-\frac12})-{f}(x_{i-\frac12})&= \sum_{k=0}^{r-1}\tilde \omega^r_k p^r_k(x_{i-\frac12})-p_0^{l_0+r-1}(x_{i-1/2})+p_0^{l_0+r-1}(x_{i-1/2})-f(x_{i-\frac12})\\
&= \sum_{k=0}^{r-1}(\tilde \omega^r_k-\hat{C}^r_k) p^r_k(x_{i-\frac12})+O(h^{r+l_0})\\
&= \sum_{k=0}^{l_0-1}(\tilde \omega^r_k-\hat{C}^r_k) (p^r_k(x_{i-\frac12})-f(x_{i-1/2}))+ \sum_{k=l_0}^{r-1}\tilde \omega^r_k p^r_k(x_{i-\frac12})+O(h^{r+l_0})\\
&=O(h^{r-1+r+1})+O(h^{2mt})+O(h^{r+l_0})\\
&=O(h^{r+l_0}).
\end{split}
\end{equation*}

\end{proof}

Finally, if $l_0=1$, then we obtain the WENO classic interpolation and the order in the interval $[x_{i},x_{i+1}]$ is $O(h^{r+1})$.

We determine the value of the parameter $t$ in (\ref{pesos}). The polynomial obtained with the new WENO technique must satisfy the following properties:
\begin{itemize}
\item It is a piecewise interpolation polynomial composed of polynomials of degree $r$.
\item Every polynomial must satisfy the following property, that is equivalent to the {\it ENO property} \cite{WENO_nuevo}, but that assures a progressive order of accuracy:
\begin{itemize}
\item The classical WENO weight related to any smooth stencil will verify $$\tilde{\omega}_{k}^r=O(1).$$
\item If the function $f$ has a singularity, then the corresponding $\tilde{\omega}_{k}^r$ will verify
\begin{equation*}
{\tilde \omega}_{k}^r= O\left(h^{r-1}\right).
\end{equation*}
\end{itemize}
\end{itemize}

\begin{corollary}\label{cor_t}
The new WENO algorithm satisfies the previous property if $t\ge r$ for jumps in the function and the first derivative.
\end{corollary}
\begin{proof}
The proof is straightforward from equation (\ref{final}) in Lemma \ref{lema_t}. We can see that assuring that $2mt\ge2r-1$ is enough. If $t=r$ the previous inequality is satisfied for jumps in the function or the first derivative. In fact, for jumps in the function it is enough if $t=ceil\left(\frac{2r-1}{2}\right)$, being {\it ceil} the operation of rounding to the upper closest integer.
\end{proof}


%

\section{Numerical experiments}\label{numexp}


This section is dedicated to present some numerical experiments aimed to test the theoretical results obtained in previous sections. We will show results related to the order of accuracy and the computational time of WENO-6, WENO-8 and WENO-10 algorithms.

The accuracy will be checked through a grid refinement analysis. We will consider the function,

\begin{equation}\label{experimento1}
f(x)=\left\{\begin{array}{ll}
x^{10}-x^9+x^8-4x^7+x^6+x^5+x^4+x^3+5x^2+3x, &a\le x<0,\\
\eta-(x^{10}-2x^9+3x^8-8x^7-2x^6+x^5-2x^4-3x^3-5x^2+0.5x),&0\le x<b,
\end{array}
\right.
\end{equation}
where we will give $\eta$ the values $\eta=0, 1$. In the first case the function presents a discontinuity in the first derivative and in the second case a discontinuity in the function. If we set $\eta=0$, then we will consider the interval $(a, b)=(-\frac{\pi}{6},1-\frac{\pi}{6})$. The reason is that we want to assure, if possible, that at all the stages of the grid refinement analysis, the singularity does not fall at a grid point. If the discontinuity does fall at a grid point, then the classical WENO strategy (or the new one) always provide an approximation of order $O(h^{r+1})$, as there is always one smooth stencil, and the grid refinement analysis does not show the real accuracy of the algorithms. If $\eta=1$, then we will just consider the interval $(a, b)=(-0.5, 0.5)$. For all the experiments we have chosen $t=r$ and $\epsilon=10^{-16}$ in (\ref{pesos}) and (\ref{pesosr}).

Tables \ref{tabla1} and \ref{tabla2} present a grid refinement analysis for the new WENO-6 and the classical WENO-6. Tables \ref{tabla3} and \ref{tabla4} present the same refinement analysis for the WENO-8 algorithms. Tables \ref{tabla5} and \ref{tabla6} present the results for the WENO-10 algorithms. All the previous tables have been obtained for the function in (\ref{experimento1}) with $\eta=0$, that presents a jump in the first derivative.

Tables \ref{tabla7} and \ref{tabla8} (WENO-6), \ref{tabla9} and \ref{tabla10} (WENO-8), \ref{tabla11} and \ref{tabla12} (WENO-10), show the same analysis but, in this case, for $\eta=1$, so (\ref{experimento1})  presents a jump in the function.

In all the aforementioned tables we use $2^i$ initial points. The errors $e_i$ are presented for interpolated data at gridpoints around the discontinuity. The interval that contains the discontinuity is always denoted as $x_{2i}$. It is clear that the accuracy is reduced step by step using the new algorithm. Classical WENO-2r algorithm is not capable of this adaption and tipically attains $O(h^{r+1})$ accuracy at the interpolations which stencils cross the discontinuity. Tables \ref{tabla3},  \ref{tabla4}, \ref{tabla5} \ref{tabla6}, \ref{tabla9}, \ref{tabla10}, \ref{tabla11} and \ref{tabla12} only present complete results to the left of the discontinuity as the accuracy is symmetric. In order to obtain the computational time, we perform $500$ executions of each subroutine and obtain the mean. The computational times obtained show that the cost is similar for both algorithms for small values of $r$, but that it grows for the new algorithm as $r$ grows.

\begin{table}[!ht]
\begin{center}
\resizebox{18cm}{!} {
\begin{tabular}{|c|c|c|c|c|c|c|c|c|c|c|c|c|c|c|c|}
\hline\multicolumn{1}{|c|}{ }&\multicolumn{2}{|c|}{$\cdots x_{2i-7}$} &\multicolumn{2}{|c|}{$x_{2i-5}$} & \multicolumn{2}{|c|}{$x_{2i-3}$} & \multicolumn{2}{|c|}{$x_{2i-1}$} & \multicolumn{2}{|c|}{$x_{2i+1}$}  & \multicolumn{2}{|c|}{$x_{2i+3}$}& \multicolumn{2}{|c|}{$x_{2i+5}\cdots$}&\multirow{ 2}{*}{Comp. t.}
              \\
\cline{1-15} $i$ &$e_i$ & $\log_2\left(\frac{e_i}{e_{i+1}}\right)$ &$e_i$ & $\log_2\left(\frac{e_i}{e_{i+1}}\right)$ & $e_i$ & $\log_2\left(\frac{e_i}{e_{i+1}}\right)$ & $e_i$ & $\log_2\left(\frac{e_i}{e_{i+1}}\right)$ & $e_i$ & $\log_2\left(\frac{e_i}{e_{i+1}}\right)$ & $e_i$ & $\log_2\left(\frac{e_i}{e_{i+1}}\right)$ & $e_i$ & $\log_2\left(\frac{e_i}{e_{i+1}}\right)$&
              \\
\hline 5&1.575e-08 & - & 7.772e-09 & - &  5.103e-06 & - &  8.277e-03  & - &  1.690e-06  & - &  3.041e-08  & - & 2.846e-08 & -&  4.214e-04
            \\
\hline 6&1.477e-10 & 6.737 & 8.152e-10 & 3.253 &  4.968e-08 & 6.683 &  9.953e-03  & -0.266 &  1.065e-07  & 3.988 &  2.929e-10  & 6.698 & 3.453e-10 & 6.365&  6.374e-04
            \\
\hline 7&1.731e-12 & 6.414 & 3.586e-11 & 4.507 &  3.319e-09 & 3.904 &  1.764e-04  & 5.818 &  1.697e-05  & -7.316 &  1.847e-08  & -5.979 & 4.733e-12 & 6.189&  1.103e-03
            \\
\hline 8&2.420e-14 & 6.160 & 1.205e-12 & 4.895 &  2.127e-10 & 3.964 &  1.764e-04  & -0.000 &  1.694e-08  & 9.969 &  1.168e-12  & 13.949 & 6.724e-14 & 6.137&   1.720e-03
            \\
\hline 9&3.539e-16 & 6.096 & 3.884e-14 & 4.955 &  1.346e-11 & 3.982 &  1.764e-04  & 0.000 &  2.981e-11  & 9.150 &  3.656e-14  & 4.997 & 1.000e-15 & 6.071&  2.788e-03
\\
\hline 10&5.204e-18 & {\bf 6.087} & 1.232e-15 & {\bf4.979} &  8.466e-13 & {\bf3.991} &  1.761e-04  & 0.003 &  1.697e-12  & {\bf4.135} &  1.198e-15  & {\bf4.931} & 1.540e-17 & {\bf 6.021}& 4.401e-03
\\
\hline
\end{tabular}
}
\caption{Grid refinement analysis for the new WENO-6 algorithm for the function in (\ref{experimento1}) and $\eta=0$. }\label{tabla1}
\end{center}
\end{table}

\begin{table}[!ht]
\begin{center}
\resizebox{18cm}{!} {
\begin{tabular}{|c|c|c|c|c|c|c|c|c|c|c|c|c|c|c|c|c|}
\hline\multicolumn{1}{|c|}{ }&\multicolumn{2}{|c|}{$\cdots x_{2i-7}$} &\multicolumn{2}{|c|}{$x_{2i-5}$} & \multicolumn{2}{|c|}{$x_{2i-3}$} & \multicolumn{2}{|c|}{$x_{2i-1}$} & \multicolumn{2}{|c|}{$x_{2i+1}$}  & \multicolumn{2}{|c|}{$x_{2i+3}$}& \multicolumn{2}{|c|}{$x_{2i+5}\cdots$}&\multirow{ 2}{*}{Comp. t.}
              \\
\cline{1-15} $i$ &$e_i$ & $\log_2\left(\frac{e_i}{e_{i+1}}\right)$ &$e_i$ & $\log_2\left(\frac{e_i}{e_{i+1}}\right)$ & $e_i$ & $\log_2\left(\frac{e_i}{e_{i+1}}\right)$ & $e_i$ & $\log_2\left(\frac{e_i}{e_{i+1}}\right)$ & $e_i$ & $\log_2\left(\frac{e_i}{e_{i+1}}\right)$ & $e_i$ & $\log_2\left(\frac{e_i}{e_{i+1}}\right)$ & $e_i$ & $\log_2\left(\frac{e_i}{e_{i+1}}\right)$&
              \\
\hline 5& 1.574e-08 & - & 2.815e-07 & - &  9.477e-06 & - &  8.190e-03  & - &  1.712e-06  & - &  3.683e-07  & - & 2.851e-08 & -&  5.273e-04
            \\
\hline 6&1.477e-10 & 6.736 & 1.170e-08 & 4.589 &  5.403e-08 & 7.454 &  9.953e-03  & -0.281 &  1.102e-07  & 3.958 &  2.454e-08  & 3.908 & 3.453e-10 & 6.367&  5.239e-04
            \\
\hline 7&1.731e-12 & 6.414 & 7.822e-10 & 3.903 &  3.320e-09 & 4.025 &  1.764e-04  & 5.818 &  2.292e-05  & -7.700 &  3.058e-07  & -3.639 & 4.732e-12 & 6.189&  1.021e-03
            \\
\hline 8&2.420e-14 & 6.160 & 4.963e-11 & 3.978 &  2.127e-10 & 3.964 &  1.764e-04  & -0.000 &  3.343e-08  & 9.421 &  8.320e-10  & 8.522 & 6.724e-14 & 6.137&  1.939e-03
            \\
\hline 9&3.539e-16 & 6.096 & 3.124e-12 & 3.990 &  1.346e-11 & 3.982 &  1.764e-04  & 0.000 &  3.263e-11  & 10.001 &  6.540e-12  & 6.991 & 1.000e-15 & 6.071&  2.191e-03
\\
\hline 10&5.204e-18 & {\bf 6.087} & 1.959e-13 & {\bf 3.995} &  8.466e-13 &{\bf 3.991} &  1.757e-04  & 0.005 &  1.698e-12  & {\bf 4.264} &  3.921e-13  & {\bf 4.060} & 1.518e-17 & {\bf 6.042}&  3.821e-03
\\
\hline
\end{tabular}
}
\caption{Grid refinement analysis for the classical WENO-6 algorithm for the function in (\ref{experimento1}) and $\eta=0$.}\label{tabla2} 
\end{center}
\end{table}

\begin{table}[!ht]
\begin{center}
\resizebox{18cm}{!} {
\begin{tabular}{|c|c|c|c|c|c|c|c|c|c|c|c|c|c|c|c|c|c|c|c|}
\hline\multicolumn{1}{|c|}{ }&\multicolumn{2}{|c|}{$\cdots x_{2i-9}$} &\multicolumn{2}{|c|}{$x_{2i-7}$} & \multicolumn{2}{|c|}{$x_{2i-5}$} & \multicolumn{2}{|c|}{$x_{2i-3}$} & \multicolumn{2}{|c|}{$x_{2i-1}$}  & \multicolumn{2}{|c|}{$x_{2i+1}$}& \multicolumn{2}{|c|}{$x_{2i+3}\cdots$}&\multirow{ 2}{*}{Comp. t.}
              \\
\cline{1-15} $i$ &$e_i$ & $\log_2\left(\frac{e_i}{e_{i+1}}\right)$ &$e_i$ & $\log_2\left(\frac{e_i}{e_{i+1}}\right)$ & $e_i$ & $\log_2\left(\frac{e_i}{e_{i+1}}\right)$ & $e_i$ & $\log_2\left(\frac{e_i}{e_{i+1}}\right)$ & $e_i$ & $\log_2\left(\frac{e_i}{e_{i+1}}\right)$ & $e_i$ & $\log_2\left(\frac{e_i}{e_{i+1}}\right)$ & $e_i$ & $\log_2\left(\frac{e_i}{e_{i+1}}\right)$&
              \\
\hline 5&3.333e-11 & - & 1.963e-09 & - &  1.881e-08 & - &  2.532e-08  & - &  7.269e-03  & - &  3.938e-08  & - & 2.776e-08 & -& 7.601e-04  
            \\
\hline 6&2.557e-13 & 7.026 & 1.275e-11 & 7.267 &  1.663e-10 & 6.821 &  2.061e-09  & 3.619 &  9.342e-03  & -0.362 &  1.133e-09  & 5.119 & 3.104e-10 & 6.483&7.292e-04 
            \\
\hline 7&6.106e-16 & 8.710 & 9.177e-14 & 7.118 &  1.713e-12 & 6.601 &  8.588e-11  & 4.585 &  1.543e-04  & 5.920 &  1.717e-06  & -10.565 & 1.735e-08 & -5.805& 1.508e-03 
            \\
\hline 8&6.939e-18 & {\bf 6.459} & 6.939e-16 & 7.047 &  2.212e-14 & 6.275 &  2.846e-12  & 4.916 &  1.543e-04  & -0.000 &  7.363e-11  & 14.509 & 4.880e-14 & 18.440& 2.334e-03   
            \\
\hline 9&0& -& 5.204e-18 & {\bf 7.059} &  3.105e-16 & 6.155 &  9.115e-14  & 4.964 &  1.543e-04  & 0.000 &  8.792e-14  & 9.710 & 6.436e-16 & 6.245& 3.927e-03 
\\
\hline 10&1.735e-18 &-& 0& -&  4.337e-18 & {\bf 6.162} &  2.880e-15  & {\bf 4.984} &  1.539e-04  & 0.004 &  2.823e-15  & {\bf 4.961} & 9.324e-18 & {\bf 6.109}& 5.832e-03
\\
\hline
\end{tabular}
}
\caption{Grid refinement analysis for the new WENO-8 algorithm for the function in (\ref{experimento1}) and $\eta=0$.}\label{tabla3} 
\end{center}
\end{table}

\begin{table}[!ht]
\begin{center}
\resizebox{18cm}{!} {
\begin{tabular}{|c|c|c|c|c|c|c|c|c|c|c|c|c|c|c|c|c|c|c|c|c|c|}
\hline\multicolumn{1}{|c|}{ }&\multicolumn{2}{|c|}{$\cdots x_{2i-9}$} &\multicolumn{2}{|c|}{$x_{2i-7}$} & \multicolumn{2}{|c|}{$x_{2i-5}$} & \multicolumn{2}{|c|}{$x_{2i-3}$} & \multicolumn{2}{|c|}{$x_{2i-1}$}  & \multicolumn{2}{|c|}{$x_{2i+1}$}& \multicolumn{2}{|c|}{$x_{2i+3}\cdots$}&\multirow{ 2}{*}{Comp. t.}
              \\
\cline{1-15} $i$ &$e_i$ & $\log_2\left(\frac{e_i}{e_{i+1}}\right)$ &$e_i$ & $\log_2\left(\frac{e_i}{e_{i+1}}\right)$ & $e_i$ & $\log_2\left(\frac{e_i}{e_{i+1}}\right)$ & $e_i$ & $\log_2\left(\frac{e_i}{e_{i+1}}\right)$ & $e_i$ & $\log_2\left(\frac{e_i}{e_{i+1}}\right)$ & $e_i$ & $\log_2\left(\frac{e_i}{e_{i+1}}\right)$ & $e_i$ & $\log_2\left(\frac{e_i}{e_{i+1}}\right)$&
              \\
\hline 5&4.966e-11 & - & 2.203e-09 & - &  4.532e-08 & - &  1.464e-07  & - &  7.246e-03  & - &  3.933e-08  & - & 2.134e-09 & -&  7.055e-04 
            \\
\hline 6&2.583e-13 & 7.587 & 1.554e-10 & 3.825 &  5.618e-10 & 6.334 &  2.064e-09  & 6.148 &  9.341e-03  & -0.366 &  1.136e-09  & 5.113 & 3.685e-10 & 2.533& 5.396e-04 
            \\
\hline 7&6.106e-16 & {\bf 8.724} & 5.924e-12 & 4.713 &  2.193e-11 & 4.679 &  8.588e-11  & 4.587 &  1.543e-04  & 5.919 &  4.457e-06  & -11.938 & 4.902e-07 & -10.377& 1.110e-03  
            \\
\hline 8&0& -& 1.924e-13 & 4.944 &  7.171e-13 & 4.934 &  2.846e-12  & 4.916 &  1.543e-04  & -0.000 &  2.158e-10  & 14.334 & 1.420e-10 & 11.753& 1.684e-03 
            \\
\hline 9&0& -& 6.113e-15 & 4.976 &  2.287e-14 & 4.971 &  9.115e-14  & 4.964 &  1.543e-04  & 0.000 &  8.977e-14  & 11.231 & 2.510e-14 & 12.466& 2.795e-03   
\\
\hline 10&1.735e-18 &-& 1.926e-16 & {\bf 4.989} &  7.225e-16 &{\bf 4.984} &  2.880e-15  &{\bf 4.984} &  1.530e-04  & 0.012 &  2.823e-15  &{\bf 4.991} & 7.082e-16 &{\bf 5.148}&4.822e-03
\\
\hline
\end{tabular}
}
\caption{Grid refinement analysis for the classical WENO-8 algorithm for the function in (\ref{experimento1}) and $\eta=0$.}\label{tabla4} 
\end{center}
\end{table}

\begin{table}[!ht]
\begin{center}
\resizebox{18cm}{!} {
\begin{tabular}{|c|c|c|c|c|c|c|c|c|c|c|c|c|c|c|c|c|c|c|c|c|c|c|c|c|c|c|}
\hline\multicolumn{1}{|c|}{ }&\multicolumn{2}{|c|}{$\cdots x_{2i-11}$} &\multicolumn{2}{|c|}{$x_{2i-9}$} & \multicolumn{2}{|c|}{$x_{2i-7}$} & \multicolumn{2}{|c|}{$x_{2i-5}$} & \multicolumn{2}{|c|}{$x_{2i-3}$}  & \multicolumn{2}{|c|}{$x_{2i-1}$}& \multicolumn{2}{|c|}{$x_{2i}\cdots$}&\multirow{ 2}{*}{Comp. t.}
              \\
\cline{1-15} $i$ &$e_i$ & $\log_2\left(\frac{e_i}{e_{i+1}}\right)$ &$e_i$ & $\log_2\left(\frac{e_i}{e_{i+1}}\right)$ & $e_i$ & $\log_2\left(\frac{e_i}{e_{i+1}}\right)$ & $e_i$ & $\log_2\left(\frac{e_i}{e_{i+1}}\right)$ & $e_i$ & $\log_2\left(\frac{e_i}{e_{i+1}}\right)$ & $e_i$ & $\log_2\left(\frac{e_i}{e_{i+1}}\right)$ & $e_i$ & $\log_2\left(\frac{e_i}{e_{i+1}}\right)$&
              \\
\hline 5&1.762e-11 & - & 1.201e-11 & - &  2.408e-11 & - &  1.603e-10  & - &  3.584e-09  & - &  5.429e-08  & - & 7.261e-03 & -& 6.430e-04 
            \\
\hline 6&3.969e-15 & {\bf12.116} & 3.136e-15 & {\bf11.903} &  2.054e-14 & 10.195 &  3.414e-13  & 8.875 &  2.284e-11  & 7.294 &  4.842e-10  & 6.809 & 9.331e-03 & -0.362& 9.058e-04    
            \\
\hline 7&2.776e-17 & 7.160 & 1.388e-17 & 7.820 &  1.388e-17 & {\bf10.531} &  9.853e-16  & {\bf8.437} &  1.651e-13  & 7.112 &  5.031e-12  & 6.589 & 1.389e-04 & 6.070&  1.897e-03   
            \\
\hline 8&1.388e-17 & 1.000 & 6.939e-18 & 1.000 &  6.939e-18 & 1.000 &  1.388e-17  & 6.150 &  1.256e-15  & 7.039 &  6.553e-14  & 6.263 & 1.389e-04 & -0.000&3.224e-03 
            \\
\hline 9&0& -& 3.469e-18 & 1.000 &  3.469e-18 & 1.000 &  1.735e-18  & 3.000 &  6.939e-18  & {\bf7.500} &  9.259e-16  & 6.145 & 1.389e-04 & 0.000&  6.051e-03    
\\
\hline 10&3.469e-18 &-& 1.735e-18 & 1.000 &  0& -&  1.735e-18  & 0.000 &  0 & -&  1.214e-17  & {\bf6.253} & 1.377e-04 & 0.012&  8.830e-03
\\
\hline
\end{tabular}
}
\caption{Grid refinement analysis for the new WENO-10 algorithm for the function in (\ref{experimento1}) and $\eta=0$.}\label{tabla5} 
\end{center}
\end{table}

\begin{table}[!ht]
\begin{center}
\resizebox{18cm}{!} {
\begin{tabular}{|c|c|c|c|c|c|c|c|c|c|c|c|c|c|c|c|c|c|c|c|c|c|c|c|c|c|c|c|}
\hline\multicolumn{1}{|c|}{ }&\multicolumn{2}{|c|}{$\cdots x_{2i-11}$} &\multicolumn{2}{|c|}{$x_{2i-9}$} & \multicolumn{2}{|c|}{$x_{2i-7}$} & \multicolumn{2}{|c|}{$x_{2i-5}$} & \multicolumn{2}{|c|}{$x_{2i-3}$}  & \multicolumn{2}{|c|}{$x_{2i-1}$}& \multicolumn{2}{|c|}{$x_{2i}\cdots$}&\multirow{ 2}{*}{Comp. t.}
              \\
\cline{1-15} $i$ &$e_i$ & $\log_2\left(\frac{e_i}{e_{i+1}}\right)$ &$e_i$ & $\log_2\left(\frac{e_i}{e_{i+1}}\right)$ & $e_i$ & $\log_2\left(\frac{e_i}{e_{i+1}}\right)$ & $e_i$ & $\log_2\left(\frac{e_i}{e_{i+1}}\right)$ & $e_i$ & $\log_2\left(\frac{e_i}{e_{i+1}}\right)$ & $e_i$ & $\log_2\left(\frac{e_i}{e_{i+1}}\right)$ & $e_i$ & $\log_2\left(\frac{e_i}{e_{i+1}}\right)$&
              \\
\hline 5&1.683e-11 & - & 1.143e-11 & - &  9.436e-10 & - &  4.340e-09  & - &  1.873e-08  & - &  6.983e-08  & - & 7.263e-03 & -& 8.538e-04  
            \\
\hline 6&3.886e-15 &{\bf 12.080} & 3.081e-15 & {\bf11.857} &  8.765e-12 & 6.750 &  3.566e-11  & 6.927 &  1.081e-10  & 7.436 &  4.843e-10  & 7.172 & 9.331e-03 & -0.362&6.782e-04 
            \\
\hline 7&1.388e-17 & 8.129 & 0& -&  9.374e-14 & 6.547 &  3.779e-13  & 6.560 &  1.135e-12  & 6.574 &  5.031e-12  & 6.589 & 1.389e-04 & 6.070&1.213e-03
            \\
\hline 8&1.388e-17 & 0.000 & 6.939e-18 &-&  1.263e-15 & 6.214 &  5.020e-15  & 6.234 &  1.493e-14  & 6.249 &  6.553e-14  & 6.263 & 1.389e-04 & -0.000& 1.894e-03 
            \\
\hline 9&0& -& 0& -&  1.735e-17 & {\bf6.186} &  7.286e-17  & {\bf6.107} &  2.134e-16  & 6.129 &  9.259e-16  & 6.145 & 1.389e-04 & 0.000&  3.582e-03 
\\
\hline 10&0& -& 1.735e-18 &-&  0& -&  1.735e-18  & 5.392 &  3.469e-18  & {\bf5.943} &  1.214e-17  & {\bf6.253} & 1.351e-04 & 0.041& 5.035e-03 
\\
\hline
\end{tabular}
}
\caption{Grid refinement analysis for the classical WENO-10 algorithm for the function in (\ref{experimento1}) and $\eta=0$.}\label{tabla6} 
\end{center}
\end{table}

%
%

\begin{table}[!ht]
\begin{center}
\resizebox{18cm}{!} {
\begin{tabular}{|c|c|c|c|c|c|c|c|c|c|c|c|c|c|c|c|}
\hline\multicolumn{1}{|c|}{ }&\multicolumn{2}{|c|}{$\cdots x_{2i-7}$} &\multicolumn{2}{|c|}{$x_{2i-5}$} & \multicolumn{2}{|c|}{$x_{2i-3}$} & \multicolumn{2}{|c|}{$x_{2i-1}$} & \multicolumn{2}{|c|}{$x_{2i+1}$}  & \multicolumn{2}{|c|}{$x_{2i+3}$}& \multicolumn{2}{|c|}{$x_{2i+5}\cdots$}&\multirow{ 2}{*}{Comp. t.}
              \\
\cline{1-15} $i$ &$e_i$ & $\log_2\left(\frac{e_i}{e_{i+1}}\right)$ &$e_i$ & $\log_2\left(\frac{e_i}{e_{i+1}}\right)$ & $e_i$ & $\log_2\left(\frac{e_i}{e_{i+1}}\right)$ & $e_i$ & $\log_2\left(\frac{e_i}{e_{i+1}}\right)$ & $e_i$ & $\log_2\left(\frac{e_i}{e_{i+1}}\right)$ & $e_i$ & $\log_2\left(\frac{e_i}{e_{i+1}}\right)$ & $e_i$ & $\log_2\left(\frac{e_i}{e_{i+1}}\right)$&
              \\
\hline 5&1.575e-08 & - & 7.924e-09 & - &  7.042e-07 & - &  4.343e-01  & - &  1.674e-06  & - &  3.041e-08  & - & 2.846e-08 & -&  4.581e-04
            \\
\hline 6&1.477e-10 & 6.737 & 8.152e-10 & 3.281 &  4.890e-08 & 3.848 &  4.579e-01  & -0.076 &  1.057e-07  & 3.985 &  2.929e-10  & 6.698 & 3.453e-10 & 6.365&  5.692e-04
            \\
\hline 7&1.731e-12 & 6.414 & 3.586e-11 & 4.507 &  3.319e-09 & 3.881 &  5.152e-01  & -0.170 &  6.714e-09  & 3.977 &  2.432e-11  & 3.590 & 4.732e-12 & 6.189&  1.214e-03
            \\
\hline 8&2.420e-14 & 6.160 & 1.205e-12 & 4.895 &  2.127e-10 & 3.964 &  5.078e-01  & 0.021 &  4.275e-10  & 3.973 &  1.046e-12  & 4.539 & 6.717e-14 & 6.139&  1.662e-03
            \\
\hline 9& 3.539e-16 & 6.096 & 3.884e-14 & 4.955 &  1.346e-11 & 3.982 &  5.041e-01  & 0.011 &  2.700e-11  & 3.985 &  3.653e-14  & 4.840 & 9.992e-16 & {\bf6.071}&  2.816e-03
\\
\hline 10&5.204e-18 & {\bf6.087} & 1.232e-15 & {\bf4.979 }&  8.466e-13 & {\bf3.991} &  5.022e-01  & 0.005 &  1.697e-12  & {\bf3.992 }&  1.221e-15  & {\bf4.903} & 1.110e-16 & 3.170&  4.345e-03
\\
\hline
\end{tabular}
}
\caption{Grid refinement analysis for the new WENO-6 algorithm for the function in (\ref{experimento1}) and $\eta=1$. }\label{tabla7}
\end{center}
\end{table}

\begin{table}[!ht]
\begin{center}
\resizebox{18cm}{!} {
\begin{tabular}{|c|c|c|c|c|c|c|c|c|c|c|c|c|c|c|c|c|}
\hline\multicolumn{1}{|c|}{ }&\multicolumn{2}{|c|}{$\cdots x_{2i-7}$} &\multicolumn{2}{|c|}{$x_{2i-5}$} & \multicolumn{2}{|c|}{$x_{2i-3}$} & \multicolumn{2}{|c|}{$x_{2i-1}$} & \multicolumn{2}{|c|}{$x_{2i+1}$}  & \multicolumn{2}{|c|}{$x_{2i+3}$}& \multicolumn{2}{|c|}{$x_{2i+5}\cdots$}&\multirow{ 2}{*}{Comp. t.}
              \\
\cline{1-15} $i$ &$e_i$ & $\log_2\left(\frac{e_i}{e_{i+1}}\right)$ &$e_i$ & $\log_2\left(\frac{e_i}{e_{i+1}}\right)$ & $e_i$ & $\log_2\left(\frac{e_i}{e_{i+1}}\right)$ & $e_i$ & $\log_2\left(\frac{e_i}{e_{i+1}}\right)$ & $e_i$ & $\log_2\left(\frac{e_i}{e_{i+1}}\right)$ & $e_i$ & $\log_2\left(\frac{e_i}{e_{i+1}}\right)$ & $e_i$ & $\log_2\left(\frac{e_i}{e_{i+1}}\right)$&
              \\
\hline 5& 1.574e-08 & - & 1.560e-07 & - &  7.042e-07 & - &  4.634e-01  & - &  1.674e-06  & - &  3.665e-07  & - & 2.851e-08 & -&  4.929e-04
            \\
\hline 6&1.477e-10 & 6.736 & 1.163e-08 & 3.746 &  4.890e-08 & 3.848 &  4.742e-01  & -0.033 &  1.057e-07  & 3.985 &  2.446e-08  & 3.906 & 3.453e-10 & 6.367&  6.711e-04
            \\
\hline 7&1.731e-12 & 6.414 & 7.821e-10 & 3.894 &  3.319e-09 & 3.881 &  5.067e-01  & -0.095 &  6.714e-09  & 3.977 &  1.560e-09  & 3.971 & 4.732e-12 & 6.189&  1.235e-03
            \\
\hline 8&2.420e-14 & 6.160 & 4.963e-11 & 3.978 &  2.127e-10 & 3.964 &  5.035e-01  & 0.009 &  4.275e-10  & 3.973 &  9.912e-11  & 3.976 & 6.728e-14 & 6.136&  1.567e-03
            \\
\hline 9& 3.539e-16 & 6.096 & 3.124e-12 & 3.990 &  1.346e-11 & 3.982 &  5.019e-01  & 0.005 &  2.700e-11  & 3.985 &  6.247e-12  & 3.988 & 8.882e-16 & {\bf6.243}&  2.527e-03
\\
\hline 10&5.204e-18 & {\bf6.087} & 1.959e-13 & {\bf3.995 }&  8.466e-13 & {\bf3.991} &  5.011e-01  & 0.002 &  1.697e-12  & {\bf3.992} &  3.920e-13  & {\bf3.994} & 1.110e-16 & 3.000&   4.038e-03
\\
\hline
\end{tabular}
}
\caption{Grid refinement analysis for the classical WENO-6 algorithm for the function in (\ref{experimento1}) and $\eta=1$.}\label{tabla8} 
\end{center}
\end{table}

\begin{table}[!ht]
\begin{center}
\resizebox{18cm}{!} {
\begin{tabular}{|c|c|c|c|c|c|c|c|c|c|c|c|c|c|c|c|c|c|c|c|}
\hline\multicolumn{1}{|c|}{ }&\multicolumn{2}{|c|}{$\cdots x_{2i-9}$} &\multicolumn{2}{|c|}{$x_{2i-7}$} & \multicolumn{2}{|c|}{$x_{2i-5}$} & \multicolumn{2}{|c|}{$x_{2i-3}$} & \multicolumn{2}{|c|}{$x_{2i-1}$}  & \multicolumn{2}{|c|}{$x_{2i+1}$}& \multicolumn{2}{|c|}{$x_{2i+3}\cdots$}&\multirow{ 2}{*}{Comp. t.}
              \\
\cline{1-15} $i$ &$e_i$ & $\log_2\left(\frac{e_i}{e_{i+1}}\right)$ &$e_i$ & $\log_2\left(\frac{e_i}{e_{i+1}}\right)$ & $e_i$ & $\log_2\left(\frac{e_i}{e_{i+1}}\right)$ & $e_i$ & $\log_2\left(\frac{e_i}{e_{i+1}}\right)$ & $e_i$ & $\log_2\left(\frac{e_i}{e_{i+1}}\right)$ & $e_i$ & $\log_2\left(\frac{e_i}{e_{i+1}}\right)$ & $e_i$ & $\log_2\left(\frac{e_i}{e_{i+1}}\right)$&
              \\
\hline 5&3.333e-11 & - & 1.963e-09 & - &  1.881e-08 & - &  4.447e-09  & - &  3.636e-01  & - &  3.940e-08  & - & 2.776e-08 & -& 5.207e-04  
            \\
\hline 6&2.557e-13 & 7.026 & 1.275e-11 & 7.267 &  1.663e-10 & 6.821 &  2.061e-09  & 1.110 &  4.152e-01  & -0.191 &  1.133e-09  & 5.120 & 3.104e-10 & 6.483& 7.497e-04  
            \\
\hline 7&6.106e-16 &{\bf 8.710} & 9.178e-14 & 7.118 &  1.713e-12 & 6.601 &  8.588e-11  & 4.585 &  5.379e-01  & -0.374 &  6.366e-11  & 4.153 & 3.810e-12 & 6.348& 1.391e-03 
            \\
\hline 8&6.939e-18 & 6.459 & 6.939e-16 & 7.047 &  2.212e-14 & 6.275 &  2.846e-12  & 4.916 &  5.194e-01  & 0.050 &  2.548e-12  & 4.643 & 4.741e-14 & 6.328&  2.160e-03  
            \\
\hline 9&0 &  - & 5.204e-18 & {\bf7.059} &  3.105e-16 & 6.155 &  9.115e-14  & 4.964 &  5.099e-01  & 0.027 &  8.682e-14  & 4.875 & 8.882e-16 & {\bf5.738}&   3.836e-03  
\\
\hline 10&1.735e-18 & -  & 0 &  - &  4.337e-18 & {\bf6.162} &  2.880e-15  & {\bf4.984} &  5.051e-01  & 0.014 &  2.887e-15  & {\bf4.911} & 1.110e-16 & 3.000&6.112e-03  
\\
\hline
\end{tabular}
}
\caption{Grid refinement analysis for the new WENO-8 algorithm for the function in (\ref{experimento1}) and $\eta=1$.}\label{tabla9} 
\end{center}
\end{table}

\begin{table}[!ht]
\begin{center}
\resizebox{18cm}{!} {
\begin{tabular}{|c|c|c|c|c|c|c|c|c|c|c|c|c|c|c|c|c|c|c|c|c|c|}
\hline\multicolumn{1}{|c|}{ }&\multicolumn{2}{|c|}{$\cdots x_{2i-9}$} &\multicolumn{2}{|c|}{$x_{2i-7}$} & \multicolumn{2}{|c|}{$x_{2i-5}$} & \multicolumn{2}{|c|}{$x_{2i-3}$} & \multicolumn{2}{|c|}{$x_{2i-1}$}  & \multicolumn{2}{|c|}{$x_{2i+1}$}& \multicolumn{2}{|c|}{$x_{2i+3}\cdots$}&\multirow{ 2}{*}{Comp. t.}
              \\
\cline{1-15} $i$ &$e_i$ & $\log_2\left(\frac{e_i}{e_{i+1}}\right)$ &$e_i$ & $\log_2\left(\frac{e_i}{e_{i+1}}\right)$ & $e_i$ & $\log_2\left(\frac{e_i}{e_{i+1}}\right)$ & $e_i$ & $\log_2\left(\frac{e_i}{e_{i+1}}\right)$ & $e_i$ & $\log_2\left(\frac{e_i}{e_{i+1}}\right)$ & $e_i$ & $\log_2\left(\frac{e_i}{e_{i+1}}\right)$ & $e_i$ & $\log_2\left(\frac{e_i}{e_{i+1}}\right)$&
              \\
\hline 5&4.966e-11 & - & 1.558e-09 & - &  4.137e-09 & - &  4.447e-09  & - &  4.455e-01  & - &  3.940e-08  & - & 2.211e-09 & -& 4.224e-04  
            \\
\hline 6&2.583e-13 & 7.587 & 1.554e-10 & 3.326 &  5.605e-10 & 2.884 &  2.061e-09  & 1.110 &  4.653e-01  & -0.063 &  1.133e-09  & 5.120 & 3.668e-10 & 2.592& 5.530e-04  
            \\
\hline 7&6.106e-16 & {\bf8.724} & 5.924e-12 & 4.713 &  2.193e-11 & 4.676 &  8.588e-11  & 4.585 &  5.111e-01  & -0.135 &  6.366e-11  & 4.153 & 1.692e-11 & 4.438& 1.070e-03  
            \\
\hline 8&0 &  - & 1.924e-13 & 4.944 &  7.171e-13 & 4.934 &  2.846e-12  & 4.916 &  5.057e-01  & 0.015 &  2.548e-12  & 4.643 & 6.493e-13 & 4.703& 1.596e-03  
            \\
\hline 9&0 & - & 6.113e-15 & 4.976 &  2.287e-14 & 4.971 &  9.115e-14  & 4.964 &  5.030e-01  & 0.008 &  8.682e-14  & 4.875 & 2.220e-14 & 4.870&  2.431e-03 
\\
\hline 10&1.735e-18 & -  & 1.926e-16 & {\bf4.989} &  7.225e-16 &{\bf 4.984} &  2.880e-15  &{\bf 4.984} &  5.016e-01  & 0.004 &  2.887e-15  & {\bf4.911} & 7.772e-16 &{\bf 4.837}& 4.763e-03 
\\
\hline
\end{tabular}
}
\caption{Grid refinement analysis for the classical WENO-8 algorithm for the function in (\ref{experimento1}) and $\eta=1$.}\label{tabla10} 
\end{center}
\end{table}

\begin{table}[!ht]
\begin{center}
\resizebox{18cm}{!} {
\begin{tabular}{|c|c|c|c|c|c|c|c|c|c|c|c|c|c|c|c|c|c|c|c|c|c|c|c|}
\hline\multicolumn{1}{|c|}{ }&\multicolumn{2}{|c|}{$\cdots x_{2i-11}$} &\multicolumn{2}{|c|}{$x_{2i-9}$} & \multicolumn{2}{|c|}{$x_{2i-7}$} & \multicolumn{2}{|c|}{$x_{2i-5}$} & \multicolumn{2}{|c|}{$x_{2i-3}$}  & \multicolumn{2}{|c|}{$x_{2i-1}$}& \multicolumn{2}{|c|}{$x_{2i}\cdots$}&\multirow{ 2}{*}{Comp. t.}
              \\
\cline{1-15} $i$ &$e_i$ & $\log_2\left(\frac{e_i}{e_{i+1}}\right)$ &$e_i$ & $\log_2\left(\frac{e_i}{e_{i+1}}\right)$ & $e_i$ & $\log_2\left(\frac{e_i}{e_{i+1}}\right)$ & $e_i$ & $\log_2\left(\frac{e_i}{e_{i+1}}\right)$ & $e_i$ & $\log_2\left(\frac{e_i}{e_{i+1}}\right)$ & $e_i$ & $\log_2\left(\frac{e_i}{e_{i+1}}\right)$ & $e_i$ & $\log_2\left(\frac{e_i}{e_{i+1}}\right)$&
              \\
\hline 5&1.762e-11 & - & 1.201e-11 & - &  2.408e-11 & - &  1.603e-10  & - &  3.585e-09  & - &  5.425e-08  & - & 2.985e-01 & -& 5.433e-04 
            \\
\hline 6&3.969e-15 & {\bf12.116} & 3.136e-15 & {\bf11.903} &  2.051e-14 & 10.197 &  3.414e-13  & 8.876 &  2.284e-11  & 7.294 &  4.842e-10  & 6.808 & 3.686e-01 & -0.304&  8.558e-04 
            \\
\hline 7&2.776e-17 & 7.160 & 1.388e-17 & 7.820 &  1.388e-17 & {\bf10.529} &  9.853e-16  & {\bf8.437} &  1.651e-13  & 7.112 &  5.031e-12  & 6.589 & 5.640e-01 & -0.614&1.669e-03  
            \\
\hline 8&1.388e-17 & 1.000 & 6.939e-18 & 1.000 &  6.939e-18 & 1.000 &  1.388e-17  & 6.150 &  1.256e-15  & 7.039 &  6.553e-14  & 6.263 & 5.329e-01 & 0.082&2.946e-03  
            \\
\hline 9&0 &  - & 3.469e-18 & 1.000 &  3.469e-18 & 1.000 &  1.735e-18  & 3.000 &  6.939e-18  & {\bf7.500} &  9.259e-16  & 6.145 & 5.167e-01 & 0.045& 5.356e-03  
\\
\hline 10&3.469e-18 & -  & 1.735e-18 & 1.000 &  0 &  - &  1.735e-18  & 0.000 &  0  &  - &  1.214e-17  & {\bf6.253} & 5.085e-01 & 0.023& 8.741e-03  
\\
\hline
\end{tabular}
}
\caption{Grid refinement analysis for the new WENO-10 algorithm for the function in (\ref{experimento1}) and $\eta=1$.}\label{tabla11} 
\end{center}
\end{table}

\begin{table}[!ht]
\begin{center}
\resizebox{18cm}{!} {
\begin{tabular}{|c|c|c|c|c|c|c|c|c|c|c|c|c|c|c|c|c|c|c|c|c|c|c|c|c|c|}
\hline\multicolumn{1}{|c|}{ }&\multicolumn{2}{|c|}{$\cdots x_{2i-11}$} &\multicolumn{2}{|c|}{$x_{2i-9}$} & \multicolumn{2}{|c|}{$x_{2i-7}$} & \multicolumn{2}{|c|}{$x_{2i-5}$} & \multicolumn{2}{|c|}{$x_{2i-3}$}  & \multicolumn{2}{|c|}{$x_{2i-1}$}& \multicolumn{2}{|c|}{$x_{2i}\cdots$}&\multirow{ 2}{*}{Comp. t.}
              \\
\cline{1-15} $i$ &$e_i$ & $\log_2\left(\frac{e_i}{e_{i+1}}\right)$ &$e_i$ & $\log_2\left(\frac{e_i}{e_{i+1}}\right)$ & $e_i$ & $\log_2\left(\frac{e_i}{e_{i+1}}\right)$ & $e_i$ & $\log_2\left(\frac{e_i}{e_{i+1}}\right)$ & $e_i$ & $\log_2\left(\frac{e_i}{e_{i+1}}\right)$ & $e_i$ & $\log_2\left(\frac{e_i}{e_{i+1}}\right)$ & $e_i$ & $\log_2\left(\frac{e_i}{e_{i+1}}\right)$&
              \\
\hline 5&1.683e-11 & - & 1.143e-11 & - &  9.434e-10 & - &  3.901e-09  & - &  1.192e-08  & - &  5.425e-08  & - & 4.077e-01 & -& 4.379e-04 
            \\
\hline 6&3.886e-15 & {\bf12.080} & 3.081e-15 & {\bf11.857} &  8.765e-12 & 6.750 &  3.566e-11  & 6.773 &  1.081e-10  & 6.785 &  4.842e-10  & 6.808 & 4.452e-01 & -0.127&6.103e-04 
            \\
\hline 7&1.388e-17 & 8.129 & 0 &  - &  9.374e-14 & 6.547 &  3.779e-13  & 6.560 &  1.135e-12  & 6.573 &  5.031e-12  & 6.589 & 5.213e-01 & -0.228& 1.043e-03
            \\
\hline 8&1.388e-17 & 0.000 & 6.939e-18 & -  &  1.263e-15 & 6.214 &  5.020e-15  & 6.234 &  1.493e-14  & 6.249 &  6.553e-14  & 6.263 & 5.109e-01 & 0.029& 1.610e-03
            \\
\hline 9&0 &  - & 0 &  - &  1.735e-17 & {\bf6.186} &  7.286e-17  & {\bf6.107 }&  2.134e-16  & 6.129 &  9.259e-16  & 6.145 & 5.056e-01 & 0.015&   3.243e-03
\\
\hline 10&0 & - & 1.735e-18 & -  &  0 &  - &  1.735e-18  & 5.392 &  3.469e-18  & {\bf5.943} &  1.214e-17  & {\bf6.253} & 5.029e-01 & 0.008& 4.713e-03  
\\
\hline
\end{tabular}
}
\caption{Grid refinement analysis for the classical WENO-10 algorithm for the function in (\ref{experimento1}) and $\eta=1$.}\label{tabla12} 
\end{center}
\end{table}

\section{Conclusions}\label{conc}
In this work we have generalized the algorithm introduced in \cite{WENO_nuevo, generalizacion} for data discretized in the point values. We have given explicit expressions for all the weights and we have proved in general that the accuracy attained with this new WENO-2r strategy is optimal for any value of $r$. We have also proposed an strategy to use the smoothness indicators of order $r$, i.e. those used by the classical WENO algorithm, as smoothness indicators of high order, using a tree structure, in order to optimize the computational cost of the new algorithm. We have presented numerical results that support the theoretical conclusions reached. We have also presented numerical estimations of the computational time that show that for low values of $r$ the new algorithm and the classical WENO algorithm perform similar.

\footnotesize

\bibliographystyle{unsrt}
\bibliography{bibliografia_bibdesk_no_doi_url}
\end{document}